\documentclass[a4paper]{amsart}

\usepackage{amsmath,amssymb}
\usepackage[dvipsnames]{xcolor}
\usepackage{float}

\newtheorem{theorem}{Theorem}[section]
\newtheorem{proposition}[theorem]{Proposition}
\newtheorem{corollary}[theorem]{Corollary}
\newtheorem{lemma}[theorem]{Lemma}

\newtheorem{preremark}[theorem]{Remark}
\newtheorem{predefinition}[theorem]{Definition}
\newtheorem{preexample}[theorem]{Example}
\newtheorem{prenotation}[theorem]{Notation}
\newtheorem{preconjecture}[theorem]{Conjecture}

\newenvironment{remark}{\begin{preremark}\rm}{\end{preremark}}
\newenvironment{definition}{\begin{predefinition}\rm}
{\end{predefinition}}
\newenvironment{example}{\begin{preexample}\rm}{\end{preexample}}

\def\OO{{\mathcal{O}}}
\def\QQ{\mathbb{Q}}

\newcommand{\M}{{\mathfrak{M}}}

\let\epsilon=\varepsilon

\def\phi{{\varphi}}
\let\Psi=\varPsi
\let\Phi=\varPhi
\let\theta=\vartheta

\def\LT{\mathop{\rm LT}\nolimits}
\def\LC{\mathop{\rm LC}\nolimits}

\def\NF{\mathop{\rm NF}\nolimits}

\def\Mat{\mathop{\rm Mat}\nolimits}

\def\Supp{\mathop{\rm Supp}\nolimits}
\def\Spec{\mathop{\rm Spec}\nolimits}

\def\GFan{\mathop{\rm GFan}\nolimits}

\def\DLog{\mathop{\rm DLog}\nolimits}

\def\Mat{\mathop{\rm Mat}\nolimits}

\def\Sep{\mathop{\rm Sep}\nolimits}

\newcommand{\Lin}{\mathop{\rm Lin}\nolimits}

\let\To=\longrightarrow
\def\TTo#1{\mathop{\longrightarrow}\limits ^{#1}}

\def\tr{^{\,\rm tr}}
\def\tfrac #1#2{{\textstyle\frac{#1}{#2}}}

\def\cocoa{\mbox{\rm
  C\kern-.13em o\kern-.07 em C\kern-.13em o\kern-.15em A}}
\def\apcocoa{\mbox{\rm
A\kern-0.13em p\kern -0.07em C\kern-.13em o\kern-.07 em C\kern-.13em
o\kern-.15em A}}

\date{\today}

\begin{document}

\title{Restricted Gr\"obner Fans and Re-embeddings of Affine Algebras}

\author{Martin Kreuzer}
\address{Fakult\"at f\"ur Informatik und Mathematik, Universit\"at
Passau, D-94030 Passau, Germany}
\email{Martin.Kreuzer@uni-passau.de}

\author{Le Ngoc Long}
\address{Fakult\"at f\"ur Informatik und Mathematik, Universit\"at
Passau, D-94030 Passau, Germany and Department of Mathematics, 
University of Education -- Hue University, 34 Le Loi Street, Hue City, Vietnam}
\email{lelong@hueuni.edu.vn, ngoc-long.le@uni-passau.de}

%    Information for third author
\author{Lorenzo Robbiano}
\address{Dipartimento di Matematica, Universit\`a di Genova,
Via Dodecaneso 35,
I-16146 Genova, Italy}
\email{lorobbiano@gmail.com}

\begin{abstract}
In this paper we continue the study of good re-embeddings of affine
$K$-algebras started in~\cite{KLR2}. The idea is to use 
special linear projections to find isomorphisms
between a given affine $K$-algebra $K[X]/I$, where $X=(x_1,\dots,x_n)$,
and $K$-algebras having fewer generators.
These projections are induced by particular tuples of indeterminates~$Z$ 
and by term orderings~$\sigma$ which realize~$Z$ as 
leading terms of a tuple~$F$ of polynomials in~$I$. In order to efficiently find such
tuples, we provide two major new tools: an algorithm which reduces the
check whether a given tuple~$F$ is $Z$-separating to an LP
feasibility problem, and an isomorphism between the part of the Gr\"obner fan
of~$I$ consisting of marked reduced Gr\"obner bases which contain a
$Z$-separating tuple and the Gr\"obner fan of $I\cap K[X\setminus Z]$.
We also indicate a possible generalization to tuples~$Z$ which consist
of terms. All results are illustrated by explicit examples. 
\end{abstract}



\maketitle

%%%%%%%%%%%%%%%%%%%%%%%%%%%%%%%%%%%%%%
%
%  Section 1: Introduction
%
%%%%%%%%%%%%%%%%%%%%%%%%%%%%%%%%%%%%%%

\section{Introduction}
\label{sec:Intro}

This paper is a natural continuation of~\cite{KLR2}. The topic treated here and there
is the search for good re-embeddings of affine algebras over a field~$K$, 
or equivalently, of affine schemes. What do we mean by that? It is a classical research topic
in algebraic geometry to find embeddings of a given scheme into low-dimensional spaces.
For affine varieties, the main result of~\cite{Sri} has been generalized in 
several directions (see for instance~\cite{SY} and~\cite{RS}, Section~10.2).
Here we follow a more computational approach which tries to avoid the frequently costly
calculation of Gr\"obner bases as much as possible.

Let $K$ be a field and~$I$ an ideal in a polynomial ring $P=K[x_1, \dots, x_n]$.
We are looking for a 
polynomial ring $P'=K[y_1,\dots, y_m]$ and an ideal~$I'$ in~$P'$ such 
that $m<n$ and such that there is a $K$-algebra isomorphism $P/I \cong P'/I'$.
In other words, we are looking for a smaller number of $K$-algebra generators of $P/I$.
In general, this problem is very hard, but there are chances to get 
good solutions using the following type of linear projections. 

Assume that $I \subseteq \langle x_1,\dots,x_n\rangle$.
Let $Z=(z_1,\dots,z_s)$ be a set of distinct indeterminates in $X=(x_1,\dots,x_n)$,
and suppose that there exist a term ordering~$\sigma$ and polynomials
$f_1,\dots,f_s \in I$ with $\LT_\sigma(f_i) = z_i$ for $i=1,\dots,s$
and such that this is the only appearance of~$z_i$ in a term of one
of the polynomials $f_1,\dots,f_s$. Then we say that~$I$ is $Z$-separating
and~$\sigma$ is a $Z$-separating term ordering for~$I$. In this setting,
the reduced $\sigma$-Gr\"obner basis of~$I$ allows us to define
a $K$-algebra isomorphism $P/I \cong \widehat{P} / (I\cap \widehat{P})$,
where $\widehat{P} = K[X \setminus Z]$, called a $Z$-separating re-embedding of~$I$.
Notice that the search for such $f_1, \dots, f_s \in I$ is in general non-trivial.
In particular, they can be hidden and far away from a given set of generators of~$I$
(see for instance~\cite{Cr} and~\cite{KLR2}, Example 3.7).

As explained in~\cite{KLR2}, the discovery of $Z$-separating 
re-embeddings for non-trivial ideals~$I$ relies on the study of the Gr\"obner 
fan $\GFan(I)$ which was introduced first in~\cite{MR}. 
As mentioned above, it may
not be possible to use a $Z$-separating re-embedding to get an optimal
re-embedding $P/I \cong P'/I'$ in the sense that~$\dim(P')$ is the embedding
dimension of~$P/I$, i.e., the smallest possible number. Moreover,
the usage of Gr\"obner fans has the added disadvantage that their
computation is prohibitively expensive in all but the smallest examples.

This brings us to the main topic of this paper, namely the task to 
improve the search for $Z$-separating re-embeddings and to use a smaller portion 
of the Gr\"obner fan which is easier to compute. To achieve this goal, we take
a markedly different point of view than in~\cite{KLR2}. Here we concentrate
not on specific $Z$-separating tuples, but on finding suitable tuples~$Z$
and the corresponding $Z$-separating tuples for a given ideal~$I$.
As a consequence, we have to study the intimate relationships between
the given ideal~$I$, possible tuples~$Z$, and possible $Z$-separating
tuples of polynomials $F=(f_1,\dots,f_s)$ with $f_i\in I$ very carefully,
and in fact the entire Section~\ref{sec:Z-Separating Ideals}
is devoted to this clarification task. In particular, Proposition~\ref{prop:Z-sep}
explains the relationship between $Z$-separating tuples and elimination orderings for~$Z$.

Another novelty relates to the search for suitable $Z$-separating tuples. 
Suppose we have a reasonable candidate for $Z=(z_1,\dots,z_s)$ 
and for a $Z$-separating tuple $F=(f_1,\dots,f_s)$ of polynomials in~$I$. 
How can we check if~$F$ is indeed a $Z$-separating tuple for~$I$?
Recall that this means that we need to find a term ordering~$\sigma$ such
that $\LT_\sigma(f_i)=z_i$ for $i=1,\dots,s$.
In Section~\ref{sec:Finding Z-Separating Tuples} we show how to convert this problem 
to a Linear Programming (LP) feasibility problem (see Proposition~\ref{prop:charZ-sep}).
The usage of LP solvers for this task is explained in detail in Corollary~\ref{cor:checkZ-sep}. 
Several examples illustrate the efficiency and power of this approach.

In Section~\ref{sec:Z-Separating Re-Embeddings}, we characterize $Z$-separating term 
orderings for~$I$ by the shape of their reduced Gr\"obner basis 
(see Proposition~\ref{prop:charZ-sepTO}). This Gr\"obner basis allows us then to construct
the desired re-embeddings of~$I$ which we called $Z$-separating
re-embeddings in~\cite{KLR2}. Altogether, we arrive at an efficient strategy for
finding good re-embeddings of~$I$ which does not require
the (possibly expensive) pre-calculation of a reduced Gr\"obner basis of~$I$:
\begin{enumerate}
\item[(1)] Find a (large) tuple of indeterminates $Z=(z_1,\dots,z_s)$ 
and a tuple of polynomials $F=(f_1,\dots,f_s)$ such that $f_i$ is
$z_i$-separating for $i=1,\dots,s$.

\item[(2)] Using an LP solver, verify that~$F$ is $Z$-separating
and obtain a $Z$-separating term ordering for~$I$.

\item[(3)] With the help of some inexpensive interreduction steps,
create  polynomials $z_1 -h_1,\dots,z_s -h_s$
in~$I$ with $h_i \in \widehat{P}= K[X\setminus Z]$.

\item[(4)] Using the polynomials $h_1,\dots,h_s$, define the $Z$-separating re-embedding
$\Phi_Z:\; P/I \longrightarrow \widehat{P} / (I \cap \widehat{P})$
of~$I$.
\end{enumerate}
The viability and efficiency of this strategy is then demonstrated
using some concrete examples.

The theoretical main result of this paper is contained in 
Section~\ref{sec:Z-Restricted Groebner Fans}. Recall that the Gr\"obner fan
$\GFan(I)$ of~$I$ consists of all marked reduced Gr\"obner bases of~$I$.
In non-trivial cases, it tends to be huge and very demanding to compute.
In Definition~\ref{def:Z-GFan} we introduce the notion 
of the $Z$-restricted Gr\"obner fan of~$I$,
denoted by $\GFan_Z(I)$, which is the set of all marked reduced Gr\"obner bases 
containing a $Z$-separating tuple of polynomials. 
Then, in Theorem~\ref{thm:restrGFan}, we prove that there is a bijective map 
$\Gamma_Z:\; \GFan_Z(I)\longrightarrow \GFan(I\cap \widehat{P})$.
Since $\widehat{P} = K[X \setminus Z]$ may have considerably fewer indeterminates 
than~$P$, this turns out to be a nice tool which can
simplify the search for a good, and possibly optimal, re-embeddings of~$I$.
Some examples illustrate this phenomenon. 

Section~\ref{sec:Computing} provides several heuristics and approaches 
for actually finding good re-embeddings of~$I$, and it explains how
to overcome some difficulties that may arise. 
In particular, Example~\ref{ex:singular} deals with the problem that
the hypothesis $I\subseteq \langle x_1,\dots,x_n\rangle$, which we were
using throughout the paper, may not be satisfied. Therefore we may need 
to perform a linear change of coordinates such that the origin is contained 
in $\mathcal{Z}(I)$. Which point should we move to the origin?
By~\cite{KLR2}, Theorem~4.1, the dimension of  the cotangent space
at the origin is a lower bound for the embedding dimension of~$P/I$.
Therefore we should move the {\it worst} singularity of~$P/I$
to the origin. Since there is a unique $K$-rational singular point
in  this example, we know what we have to do, but the general situation
may be more challenging. 
On the positive side, at the end of this section we also
provide a criterion which allows us to show that some re-embeddings are
actually isomorphisms between the given scheme and an affine space 
(see Proposition~\ref{prop:AffineSpace}).

Finally, in Section~\ref{sec:T-RestrGFan} we generalize the approach 
from using $Z$-separating tuples of polynomials to $T$-separating tuples,
where $T=(t_1,\dots,t_s)$ denotes a tuple of terms that we try to
realize as leading terms of polynomials in~$I$. To adapt the definitions
to this more general setting, we let~$Z$ be the tuple of indeterminates
dividing one of the terms in~$T$ and $Y=X\setminus Z$.
Using a suitable definition of a $T$-separating
Gr\"obner fan $\GFan_T(I)$, we show that there is a free module~$M$ 
over~$K[Y]$ such that the elements of $\GFan_T(I)$ are related to
$K[Y]$-module Gr\"obner bases of $I\cap M$ and such that
a $T$-separating module re-embedding $P/I \cong M/(I\cap M)$ of~$I$ results.
However, we were not able to find an analogue of Theorem~\ref{thm:restrGFan}
in this setting and leave this task for future research.

True to our preferred style, we have sprinkled this paper generously
with many illustrative examples. The calculations underlying these examples 
were performed using the computer algebra system \cocoa\ (see~\cite{CoCoA}) 
and with the help  of the several \cocoa -packages written by the second 
and third authors. For the notation and definitions used throughout the paper, 
we follow~\cite{KR1} and~\cite{KR2}.

\bigbreak
%%%%%%%%%%%%%%%%%%%%%%%%%%%%%%%%%%%%%%%%%%%%%%%%%%%%%%%%
%
% Section 2: Z-Separating Ideals
%
%%%%%%%%%%%%%%%%%%%%%%%%%%%%%%%%%%%%%%%%%%%%%%%%%%%%%%%%%

\section{$Z$-Separating Polynomials, Tuples, and Ideals}
\label{sec:Z-Separating Ideals}

In this section we use the notation introduced in~\cite{KLR2} with some 
appropriate changes and extensions. Specifically, we let $K$ be a field, 
let $P=K[x_1, \dots, x_n]$, and let $\M = \langle x_1, \dots, x_n\rangle$. 
The tuple formed by the indeterminates of~$P$ is denoted by $X = (x_1,\dots, x_n)$. 
Furthermore, we let $1\le s\le n$, let $z_1, \dots, z_s$ be pairwise 
distinct indeterminates in~$X$, and let $Z =(z_1, \dots, z_s)$.
The remaining indeterminates are denoted by 
$\{y_1,\dots,y_{n-s}\} = \{x_1,\dots,x_n\} \setminus \{z_1,\dots,z_s\}$,
and we let $Y =(y_1,\dots, y_{n-s})$.
Finally, given a term ordering $\sigma$ on~$P$,
its restriction to $K[Y]=K[y_1, \dots, y_{n-s}]$ is denoted by~$\sigma_Y$.

The following definition extends~\cite{KLR2}, Definitions~2.1 and~2.5.

\begin{definition}\label{def:Z-sep}
In the above setting, let $f_1,\dots,f_s \in\mathfrak{M} \setminus \{0\}$,
let $F=(f_1, \dots, f_s)$, and let $I_F = \langle f_1,\dots,f_s\rangle$
be the ideal generated by $\{f_1,\dots,f_s\}$.
\begin{enumerate}
\item[(a)] Given $i\in \{1,\dots,s\}$, we say that the polynomial~$f_i$ is
{\bf $z_i$-separating} if $z_i \in \Supp(f_i)$ and~$z_i$ does not divide
any other term in $\Supp(f_i)$.

\item[(b)] The tuple~$F$ is called {\bf $Z$-separating} if there exists 
a term ordering~$\sigma$ such that $\LT_\sigma(f_i) = z_i$ for $i=1,\dots,s$.
In this case~$\sigma$ is called a {\bf $Z$-separating term ordering} for~$F$.

\item[(c)] The tuple~$F$ is called {\bf coherently $Z$-separating} if it is
$Z$-separating, i.e., there exists a term ordering~$\sigma$ such that 
$\LT_\sigma(f_i)=z_i$ for $i=1,\dots,s$,
and if, additionally, the reduced $\sigma$-Gr\"obner basis of~$I_F$ is 
$\{\frac{1}{c_1}\, f_1,\dots, \frac{1}{c_s}\, f_s \}$, 
where $c_i = \LC_\sigma(f_i)$ for $i=1,\dots,s$.

\item[(d)] The ideal $I_F$ is called {\bf $Z$-separating} if there exists a 
term ordering~$\sigma$ such that $\LT_\sigma(I_F) = \langle Z\rangle$.
In this case~$\sigma$ is called a {\bf $Z$-separating term ordering} for~$I_F$.

\item[(e)] The set of all $Z$-separating term orderings for~$I_F$ is
denoted by $\Sep_Z(I_F)$.
\end{enumerate}
\end{definition}

Let us collect some basic observations about these notions.

\begin{proposition}\label{prop:Z-sep}
In the above setting, let $f_1,\dots,f_s \in\mathfrak{M} \setminus \{0\}$,
let $F=(f_1, \dots, f_s)$, and let $I_F = \langle f_1,\dots,f_s\rangle$.
\begin{enumerate}
\item[(a)] If~$F$ is a $Z$-separating tuple then $f_i$ is $z_i$-separating
for $i=1,\dots,s$.

\item[(b)] The tuple~$F$ is coherently $Z$-separating if and only
if $f_i$ is $z_i$-separating for $i=1,\dots,s$ and~$z_i$ does not divide any term 
in~$\Supp(f_j)$ for $i,j\in \{1,\dots,s\}$ such that $j\ne i$.

\item[(c)] If the ideal $I_F$ is $Z$-separating and~$\sigma$ is a
$Z$-separating term ordering for~$I_F$ then the reduced $\sigma$-Gr\"obner
basis of~$I_F$ is coherently $Z$-separating.

\item[(d)] If~$F$ is a $Z$-separating tuple and $\sigma$ is a $Z$-separating
term ordering for~$F$ then the ideal $I_F$ is $Z$-separating and~$\sigma$
is a $Z$-separating term ordering for~$I_F$. Moreover, the reduced $\sigma$-Gr\"obner
basis of~$I_F$ is obtained by making the polynomials in~$F$ monic and interreducing them.

\item[(e)] If the ideal~$I_F$ is~$Z$-separating and~$\sigma$ is an elimination 
ordering for~$Z$ then $\sigma\in \Sep_Z(I_F)$.

\end{enumerate}
\end{proposition}

\begin{proof}
To prove~(a), we let $i\in \{1,\dots,s\}$. Then $z_i$ does not divide
any term in~$\Supp(f_i)$ besides itself, since in this case that term would
be larger than~$z_i$ w.r.t.~$\sigma$ in contradiction to $z_i = \LT_\sigma(f_i)$.

Claim~(b) was proven in~\cite{KLR2}, Proposition 2.6.
Next we show~(c). Using Definition~\ref{def:Z-sep}.d, we know that there
exists polynomials $g_1,\dots,g_s\in I_F$ such that $\LT_\sigma(g_i)=z_i$
for $i=1,\dots,s$. In particular, this implies that~$g_i$ is $z_i$-separated.
Therefore the reduced $\sigma$-Gr\"obner basis of~$I_F$ is necessarily of the 
form $\{z_1-h_1, \dots, z_s -h_s\}$ with $h_1, \dots, h_s \in K[Y]$.
Hence the claim follows from Definition~\ref{def:Z-sep}.c.

In order to prove~(d), we note that the leading terms $z_i = \LT_\sigma(f_i)$
are pairwise coprime. Hence~\cite{KR1}, Corollary 2.5.10,  
implies that~$F$ is a $\sigma$-Gr\"obner basis of~$I_F$ and $\LT_\sigma(I_F) 
= \langle \LT_\sigma(f_1),\dots, \LT_\sigma(f_s)\rangle = \langle Z \rangle$.
The additional claim follows from the observation that~$F$ is actually a
minimal $\sigma$-Gröbner basis of~$I_F$, i.e., that $\{ z_1,\dots,z_s\}$
is the minimal monomial system of generators of~$\LT_\sigma(I_F)$.

Finally, we show~(e). By~(d), there exists a term ordering~$\tau$ such that
the reduced $\tau$-Gr\"obner basis of~$I_F$ has the form $\{z_1-h_1,
\dots, z_s - h_s\}$ with $h_1,\dots,h_s \in K[Y]$.
Since~$\sigma$ is an elimination ordering for~$Z$, it follows that 
$\LT_\sigma(z_i-h_i)=z_i$ for $i=1,\dots,s$, and hence that this is the
reduced $\sigma$-Gr\"obner basis of~$I_F$, as well.
\end{proof}

The next examples illustrate the preceding definition and proposition.
In particular, they show that a tuple of $Z$-separating polynomials is not necessarily 
a $Z$-separating tuple, and that a tuple generating a $Z$-separating
ideal is in general not a $Z$-separating tuple.

\begin{example}
Let $P=\mathbb{Q}[x,y,z]$, let $Z=(x,y)$, and let $F=(f_1,f_2)$, where
$f_1 = x-y^2$ and $f_2 = y-xz$. Then $f_1$ is $x$-separated and $f_2$ is
$y$-separated, but~$F$ is not $Z$-separated. To see why this is so,
assume that~$\sigma$ is a term ordering such that $x >_\sigma y^2$ and $y >_\sigma xz$.
Then $x >_\sigma y^2 >_\sigma x^2 z^2$ yields a contradiction to $xz^2 >_\sigma 1$.
Hence the implication in part~(a) of the proposition is strict.
\end{example}

\begin{example}
Let $P=\mathbb{Q}[x,y,z]$, let $Z=(x,y)$, and let $F=(f_1,f_2)$, where
$f_1=x-z^2$ and $f_2=y-xz$. Then $I_F = \langle f_1,f_2\rangle$ is a
$Z$-separated ideal for any term ordering~$\sigma$ satisfying
$x>_\sigma z^2$ and $y >_\sigma xz$, since then we have $\LT_\sigma(f_1)=x$
and $\LT_\sigma(f_2)=y$. Such term ordering are easy to construct, for instance
by taking $\sigma = {\rm ord}(V)$ with a matrix $V\in \Mat_3(\mathbb{Z})$
whose first row is $(3, 5, 1)$.
Thus $\{f_1,f_2\}$ is a $\sigma$-Gr\"obner basis
of~$I_F$ and $\LT_\sigma(I_F) = \langle Z\rangle$.

Note that the tuple~$F$ is not coherently $Z$-separating, because $x=\LT_\sigma(f_1)$
divides a term in $\Supp(f_2)$. To create a coherently $Z$-separating tuple,
we have to calculate the reduced $\sigma$-Gr\"obner basis of~$I_F$.
Interreducing~$f_1$ and~$f_2$ yields $\{x-z^2,\, y-z^3\}$, a coherently $Z$-separating
tuple which generates~$I_F$.
\end{example}

\begin{example}\label{ex:SumAndDiff}
Let $P=\mathbb{Q}[x,y,z]$, let $Z=(x,y)$, and let $F=(f_1,f_2)$, where
$f_1 = x+y -z^2$ and $f_2 = x - y +z^2$. Here~$F$ is not a $Z$-separating
tuple, because $x = \LT_\sigma(f_1)$ implies $x>_\sigma y$ and $y=\LT_\sigma(f_2)$
implies $y >_\sigma x$, a contradiction.

However, the ideal~$I_F$ is $Z$-separating. Let us replace the generating
tuple~$F$ by $G = (f_1+f_2,\; f_1-f_2) = (2x,\; 2y - 2z^2)$. Then~$G$ is
a $\sigma$-Gr\"obner basis of~$I_F$ for any term ordering~$\sigma$ such that
$y >_\sigma z^2$. Notice that, in order to pass from the given tuple of
generators of~$I_F$ to a coherently separating one, we had to perform
linear combinations of the given generators.
\end{example}

Furthermore, let us reconsider  Example~2.8 of~\cite{KLR2}.

\begin{example}\label{ex:toyexample}
Let $P = \mathbb{Q}[x,y,z]$, let $Z = (y, z)$, and let $F=(f_1,f_2)$, where
$f_1 = x^2 -x -y$ and $f_2 = y^2 - z$.
Then the tuple~$F$ is not coherently $Z$-separating.
However, if we reduce~$f_2$ using~$f_1$, i.e., if we replace~$y^2$ in~$f_2$ 
with $(x^2 -x)^2$, we get $\tilde{f}_2 = (x^2 -x)^2-z$. 

Then $(f_1,\tilde{f}_2) = (y-x^2+x,\ z- (x^2-x)^2)$ is the reduced $\sigma$-Gr\"obner
basis of~$I_F$ for every elimination ordering~$\sigma$ for~$Z$, and therefore
every elimination ordering for~$Z$ is in $\Sep_Z(I_F)$.
\end{example}

The process of finding and verifying $Z$-separating tuples is, in general,
very time consuming. The following remark contains some suggestions
for speeding it up.

\begin{remark}\label{usingStronglySep}
Suppose we are given polynomials $f_1,\dots,f_s \in \M \setminus \{0\}$
which generate an ideal $I_F$, and we are looking for a $Z$-separating
generating tuple of~$I_F$. 
\begin{enumerate}
\item[(a)] The first thing to do is, of course, to check whether $F=(f_1,\dots,f_s)$ 
is a $Z$-separating tuple already. It is easy to inspect $f_1,\dots,f_s$ 
and see whether they are $Z$-separating polynomials. But for a $Z$-separating tuple, 
we also need a term ordering~$\sigma$ such that $\LT_\sigma(f_i) = z_i$
for $i=1,\dots,s$. In the next section we shall examine this problem
more closely and show that it amounts to the feasibility problem of a
Linear Programming Feasibility problem which may frequently be solved quickly.

\item[(b)] Suppose that our initial test says that~$F$ is not $Z$-separated
for any term ordering. This does not exclude the possibility that there exists 
another tuple which generates~$I_F$ and is $Z$-separating.
According to Proposition~\ref{prop:Z-sep}.c, to find such a tuple
we should, in principle, choose a candidate
term ordering~$\sigma$ and compute a $\sigma$-Gr\"obner basis of~$I_F$.
For larger examples, this may take a long time. However, if the given polynomials
yield a $\sigma$-Gr\"obner basis after some interreduction steps,
the Buchberger algorithm may terminate quickly enough. Hence we can try to
choose a promising candidate term ordering~$\sigma$ and start to calculate
the $\sigma$-Gr\"obner basis of~$I_F$ with a suitable temporal or spacial timeout.
If we are lucky and the process stops, we can usually move from the resulting
Gr\"obner basis to the reduced Gr\"obner basis with comparatively little
effort and end up with a coherently $Z$-separated tuple generating~$I_F$.

\end{enumerate}
\end{remark}

\bigskip\bigbreak
%%%%%%%%%%%%%%%%%%%%%%%%%%%%%%%%%%%%%%%%%%%%%%%%%%%%%%%%
%
% Section 3: Finding Z-Separating Tuples
%
%%%%%%%%%%%%%%%%%%%%%%%%%%%%%%%%%%%%%%%%%%%%%%%%%%%%%%%%%

\section{Finding $Z$-Separating Tuples}
\label{sec:Finding Z-Separating Tuples}

In the setting of the preceding section, let $f_1,\dots,f_s \in \M\setminus \{0\}$,
let $F = (f_1,\dots,f_s)$, and let $I_F = \langle f_1,\dots,f_s\rangle$.
In the following we let $\mathbb{N}_+$ be the set of positive integers.
Recall that, for a term $t = x_1^{\alpha_1} \cdots x_n^{\alpha_n} \in\mathbb{T}^n$,
the number $\log(t) =(\alpha_1, \dots, \alpha_n)$ is called the {\bf logarithm} of~$t$.

\begin{definition}\label{def:DLog}
As above, let $f_1,\dots,f_s\in \M\setminus \{0\}$, let $F=(f_1, \dots, f_s)$,  
and assume that $z_i \in \Supp(f_i)$ for $i=1,\dots, s$. 
\begin{enumerate}
\item For $i=1,\dots,s$, the set 
$\{ \log(z_i)-\log(t) \mid t \in \Supp(f_i)\setminus \{z_i\}\}$ is 
called the {\bf set of logarithmic differences} of~$f_i$ with respect to~$z_i$ 
and is denoted by $\DLog_{z_i}(f_i)$.

\item The set $\bigcup_{i=1}^s \DLog_{z_i}(f_i)$ is called the 
{\bf set of logarithmic differences}
of~$F$ with respect to~$Z$ and is denoted by $\DLog_Z(F)$.
\end{enumerate}
\end{definition}

In the setting of Example~\ref{ex:toyexample}, we can evaluate these sets as follows.

\begin{example}\label{ex:toyexample-continued}
As in Example~\ref{ex:toyexample}, let $P=\mathbb{Q}[x,y,z]$, let $Z=(y,z)$, and let
$F = (f_1, f_2)$, where $f_1= x^2 -x -y$ and $f_2 = y^2 - z$.
Then we have $\DLog_y(f_1) = \DLog_y(x^2 -x ) = \{(-2, 1,0), (-1,1,0)\}$ and
$\DLog_z(f_2) = \DLog_z(y^2) = \{(0, -2, 1) \}$. Altogether, we obtain
$\DLog_Z(F) = \{(-2, 1,0), (-1,1,0), (0, -2, 1)\}$.
\end{example}

Recall that we may identify a term ordering on~$\mathbb{T}^n$ with
the corresponding term ordering on~$\mathbb{N}^n$ and its unique
extension to~$\mathbb{Z}^n$ as explained in~\cite{KR1}, p.\ 54.
Using this terminology, we can characterize $Z$-separating tuples as follows.

\begin{proposition}{\bf (Characterizing $Z$-Separating Tuples)}\label{prop:charZ-sep}\\
Let $f_1,\dots,f_s\in \M\setminus \{0\}$, let $F=(f_1, \dots, f_s)$, 
and assume that $z_i \in \Supp(f_i)$ for $i=1,\dots, s$. 
Then the following conditions are equivalent.
\begin{enumerate}
\item[(a)] The tuple $F$ is $Z$-separating.

\item[(b)] There exists a term ordering $\sigma$ on~$\mathbb{N}^n$ such that 
its unique extension to~$\mathbb{Z}^n$ satisfies $v>_\sigma 0$ 
for every $v \in \DLog_Z(F)$.

\item [(c)] There exists a vector $u \in \mathbb{N}_+^n$
such that $u\cdot v > 0$ for every $v \in \DLog_Z(F)$.
(Here $u\cdot v$ denotes the dot product of two elements of~$\mathbb{Z}^n$.)

\end{enumerate}
\end{proposition}

\begin{proof}
First we prove that~(a) implies~(b). By the definition of a $Z$-separating set,
we know that there exists a term ordering~$\sigma$ on~$\mathbb{T}^n$ such 
that $z_i =\LT_\sigma(f_i)$ for $i=1,\dots, s$. This implies $\DLog_{z_i}(t)
>_\sigma 0$ for every $t\in \Supp(f_i) \setminus \{z_i\}$, and therefore
$v >_\sigma 0$ for every $v\in \DLog_Z(F)$.

Next we note that~(b)$\Rightarrow$(c) is shown in the theory of Gr\"obner fans
(see~\cite{MR}, Corollary 2.2).
Finally, to prove that~(c) implies~(a), we observe that the vector~$u$ 
can be taken as the first row of a matrix $V\in \Mat_n(\mathbb{Z})$ 
which defines a term ordering $\sigma$ on~$\mathbb{T}^n$.
Then we have $v >_\sigma 0$ for all $v\in \DLog_Z(F)$, and therefore
$z_i >_\sigma t$ for all $t\in \Supp(f_i) \setminus \{z_i\}$ and $i=1,\dots,s$.
Hence we conclude that $z_i =\LT_\sigma(f_i)$ for $i=1,\dots,s$.
\end{proof}

In the next step, we transform the problem of finding the vector~$u$
in part~(c) of the preceding proposition into a Linear Programming
Feasibility (LPF) problem over the field~$\mathbb{Q}$. In this way we will
get an algorithm which can detect efficiently whether a tuple~$F$ is $Z$-separating.

\begin{definition}
Let $A = (a_{ij}) \in\Mat_{k,n}(\mathbb{Q})$ and let $b=(b_1,\dots,b_k) \in
\mathbb{Q}^k$. Consider the following system of inequalities:
\begin{align*}
a_{i1} x_1 + \cdots + a_{in} x_n \le b_i & \qquad\hbox{\rm for\ } i=1,\dots,k\\
x_j \ge 0 & \qquad\hbox{\rm for\ } j=1,\dots,n
\end{align*}
Then the task of deciding whether this system of linear inequalities has a solution
$x = (x_1,\dots,x_n)\in \mathbb{Q}^n$ is called a 
{\bf Linear Programming Feasibility (LPF)} problem in standard form over the rationals.
It is usually written as $Ax \le b$ and $x\ge 0$.
\end{definition}

It is known that LPF instances can be decided in polynomial time, including
the computation of a solution instance in the positive case (see~\cite{Kar},
\cite{Kha}). The following corollary allows us to view the task of checking whether
the tuple~$F$ is $Z$-separating as an LPF instance.

\begin{corollary}\label{cor:checkZ-sep}
Given~$Z$ and~$F$ as above, we let $\DLog_Z(F) = \{v_1,\dots,v_k\}$
and write $v_i = (v_{i1}, \dots,v_{in})$ for $i=1,\dots,k$,
where $v_{ij}\in\mathbb{Z}$. Let $A=(a_{ij})$, where $a_{ij}=-v_{ij}$
and $b=(b_1,\dots,b_k)$, where $b_i = v_{i1} + \cdots +v_{in} -1$ for
$i=1,\dots,k$. Then the following conditions are equivalent:
\begin{enumerate}
\item[(a)] The tuple~$F$ is $Z$-separating.

\item[(b)] The LPF in standard form given by $Ax \le b$ and $x\ge 0$ has
a solution $x=(x_1,\dots,x_n)$ in~$\mathbb{Q}^n$.
\end{enumerate}
If these conditions are satisfied, we let $\tilde{u}_i = x_i+1$ for $i=1,\dots,n$,
multiply the numbers $\tilde{u}_i\in\mathbb{Q}_+$ with their common denominator 
to get $u_i\in \mathbb{N}_+$,
and obtain a vector $u=(u_1,\dots,u_n) \in \mathbb{N}_+^n$ such that
$u\cdot v > 0$ for all $v\in \DLog_Z(F)$.
\end{corollary}

\begin{proof}
By the proposition, condition~(a) is equivalent to the existence of
integers $u_i\in\mathbb{N}_+$ for $i=1,\dots,n$ such that
$u_1 v_{i1} + \cdots + u_n v_{in} \ge 1$ for $i=1,\dots,k$.
Clearly, this is equivalent to the existence of rational numbers $u_i \ge 1$
satisfying these inequalities, as we can multiply those numbers by their 
common denominator and still get a solution.

Now we let $A=(a_{ij})$ with $a_{ij} = -v_{ij}$ and $b_i = v_{i1} + \cdots +
v_{in} - 1$ and $x_j = u_j -1$ for $i=1,\dots,k$ and $j=1,\dots,n$.
Then $u_1 v_{i1} + \cdots + u_n v_{in} \ge 1$ is equivalent to
$x_1 (-v_{i1}) + \cdots + x_n (-v_{in}) \le v_{i1} + \cdots + v_{in} - 1 = b_i$
for $i=1,\dots,k$. Hence the task is equivalent to finding $x=(x_1,\dots,x_n)\in  
\mathbb{Q}^n$ such that $Ax \le b$ and $x\ge 0$.

The additional claim follows immediately from these equivalences.
\end{proof}

Let us show some examples which illustrate this corollary.

\begin{example}\label{ex:smallex}
Let $P = \mathbb{Q}[x,y,z]$, let $Z=(y,z)$,
and let $F=(f_1,f_2)$, where
$f_1=x^2-x-y$ and $f_2=y^2 +x^8 -z$.
To check whether~$F$ is $Z$-separating, we form $\DLog_Z(F) = 
\{ (-1,1,0),\; (-2,1,0),\; (-8,0,1),\; (0,-2,1) \}$ and obtain the
LPF problem $Av\le b$ and $v\ge 0$ where
$$
A \;=\; \begin{pmatrix}
1 & -1 & 0\\  2 & -1 & 0\\ 8 & 0 & -1\\ 0 & 2 & -1
\end{pmatrix}\quad\hbox{\rm and}\quad
b \;=\; \begin{pmatrix}
-1\\ -2\\ -8\\ -2 
\end{pmatrix}
$$
Using an LP solver, we find that $v=(0,2,8)$ satisfies these inequalities. 
Therefore $u=(1,3,9)$ is a vector which shows that~$F$ is $Z$-separating,
and every term ordering $\sigma = {\rm ord}(V)$ for which the first
row of $V\in\Mat_3(\mathbb{Z})$ is $(1,3,9)$ will be a $Z$-separating term ordering
for~$F$.

Let us also check whether the tuple $F'=(f_1,f_2,f_3)$ is $Z'$-separating,
where $f_3= yz - x - y - z$ and $Z'=(y,z,x)$. For this purpose, we calculate
$\DLog_{Z'}(F') = \DLog_Z(F) \cup \{ (1,-1,0),\; (1,0,-1),\; (1,-1,-1)\}$
and have to solve $A'v\le b'$ and $v\ge 0$ for
$$
A' \;=\; \begin{pmatrix}
1 & -1 & 0\\  2 & -1 & 0\\ 8 & 0 & -1\\ 0 & 2 & -1\\ -1 & 1 & 0\\ -1 & 0 &1\\ -1&1&1
\end{pmatrix}\quad\hbox{\rm and}\quad b\;=\; \begin{pmatrix}
-1\\ -2\\ -8\\ -2\\ -1\\ -1\\ -2
\end{pmatrix}
$$
Here the LP solver tells us that the problem is infeasible, whence $F'$ is not 
a $Z'$-separating tuple.
\end{example}

The next example is bigger and shows that changing~$Z$ for a fixed tuple~$F$ can
lead to different outcomes.

\begin{example}\label{ex:sixteen}
Let $P = \mathbb{Q}[x_1, x_2, \dots, x_{16}]$, and let 
$F = (f_1, f_2, f_3, f_4, f_5, f_6)$, where
\begin{alignat*}{3}
f_1 &= x_{6} x_{9} +x_{11} x_{14} -x_{12}         
    &\qquad f_2 &= x_{6} x_{9} +x_{8} x_{13} -x_{7} \\
f_3 &= x_7 x_{13} + x_{12} x_{14} +x_{10} -x_{16} 
    &\qquad f_4 &= x_7 x_{14} +x_{13} x_{16} +x_5 -x_{15}\\
f_5 &= x_{13} x_{14} +x_{12} x_{15} +x_{1}        
    &\qquad f_6 &= x_{2} x_{3} -x_{4} x_{5} -x_{10} +x_{11}
\end{alignat*}
The linear parts of these polynomials are 
$$
-x_{12},\;  -x_7,\;  x_{10} -x_{16},\; x_{5} -x_{15},\; x_{1},\hbox{\ \rm and\ }-x_{10}+x_{11}
$$
We see that each polynomial in~$F$ is separated with respect to every indeterminate 
in its linear part. The question is to find out if the polynomials in~$F$ are coherently 
$Z$-separated, but of course there are many different choices for~$Z$. 

For instance, let us choose  $Z = (x_{12}, x_7, x_{16}, x_{15}, x_1, x_{10})$
and apply the algorithm described in Corollary~\ref{cor:checkZ-sep}.
First we calculate $\DLog_{Z}(F)$ and the corresponding matrix $A$, and get 
$$
A= \begin{pmatrix}
  0&  0&  0&  0&  0&  1&  0&  0&  1&  0&  0&  -1&  0&  0&  0&  0 \\
  0&  0&  0&  0&  0&  0&  0&  0&  0&  0&  1&  -1&  0&  1&  0&  0  \\
  0&  0&  0&  0&  0&  1&  -1&  0&  1&  0&  0&  0&  0&  0&  0&  0  \\
  0&  0&  0&  0&  0&  0&  -1&  1&  0&  0&  0&  0&  1&  0&  0&  0  \\
  0&  0&  0&  0&  0&  0&  1&  0&  0&  0&  0&  0&  1&  0&  0&  -1  \\
  0&  0&  0&  0&  0&  0&  0&  0&  0&  0&  0&  1&  0&  1&  0&  -1  \\
  0&  0&  0&  0&  0&  0&  0&  0&  0&  1&  0&  0&  0&  0&  0&  -1  \\
  0&  0&  0&  0&  0&  0&  1&  0&  0&  0&  0&  0&  0&  1&  -1&  0  \\
  0&  0&  0&  0&  0&  0&  0&  0&  0&  0&  0&  0&  1&  0&  -1&  1   \\
  0&  0&  0&  0&  1&  0&  0&  0&  0&  0&  0&  0&  0&  0&  -1&  0   \\
  -1&  0&  0&  0&  0&  0&  0&  0&  0&  0&  0&  0&  1&  1&  0&  0   \\
  -1&  0&  0&  0&  0&  0&  0&  0&  0&  0&  0&  1&  0&  0&  1&  0   \\
  0&  1&  1&  0&  0&  0&  0&  0&  0&  -1&  0&  0&  0&  0&  0&  0   \\
  0&  0&  0&  1&  1&  0&  0&  0&  0&  -1&  0&  0&  0&  0&  0&  0   \\
  0&  0&  0&  0&  0&  0&  0&  0&  0&  -1&  1&  0&  0&  0&  0&  0
\end{pmatrix}
$$
The matrix $b$ is 
$$
b = ( -2, -2, -2, -2, -2, -2, -1, -2, -2, -1, -2, -2, -2, -2, -1)\tr 
$$
Using an LP solver, we find that 
$x = (10, 0, 0, 0, 0, 0, 2, 0, 0, 2, 0, 2, 0, 0, 6, 4)$ 
solves $Ax\le b$ and $x\ge 0$.  
Therefore $u = (11, 1, 1, 1, 1, 1, 3, 1, 1, 3, 1, 3, 1, 1, 7, 5)$ 
satisfies $u\cdot v >0$ for all $v \in \DLog_{Z}(F)$. Let~$\sigma$ be a term 
ordering represented by a matrix whose first row is~$u$. Then
the polynomials in the reduced $\sigma$-Gr\"obner basis of~$I_F$ are
$$
\begin{array}{l}
x_{12} - (x_{6} x_{9} +x_{11} x_{14}) \cr
x_{7} \ - (x_{6} x_{9} +x_{8} x_{13}) \cr
x_{16} -(x_{6} x_{9} x_{13} +x_{8} x_{13}^2  +x_{6} x_{9} x_{14} +x_{11} x_{14}^2  
+x_{2} x_{3} -x_{4} x_{5} +x_{11}) \cr
x_{15} - (x_{6} x_{9} x_{13}^2  +x_{8} x_{13}^3  +x_{6} x_{9} x_{13} x_{14} 
+x_{11} x_{13} x_{14}^2  +x_{2} x_{3} x_{13} -x_{4} x_{5} x_{13} \cr 
\qquad  +x_{6} x_{9} x_{14} +x_{8} x_{13} x_{14} 
+x_{11} x_{13} +x_{5}) \cr
x_{1} \ -( -x_{6}^2 x_{9}^2 x_{13}^2  -x_{6} x_{8} x_{9} x_{13}^3  
-x_{6}^2 x_{9}^2 x_{13} x_{14} -x_{6} x_{9} x_{11} x_{13}^2 x_{14} 
-x_{8} x_{11} x_{13}^3 x_{14} \cr
\qquad -2 x_{6} x_{9} x_{11} x_{13} x_{14}^2  -x_{11}^2 x_{13} x_{14}^3  
-x_{2} x_{3} x_{6} x_{9} x_{13} +x_{4} x_{5} x_{6} x_{9} x_{13} -x_{6}^2 x_{9}^2 x_{14} \cr
\qquad -x_{6} x_{8} x_{9} x_{13} x_{14} -x_{2} x_{3} x_{11} x_{13} x_{14} 
+x_{4} x_{5} x_{11} x_{13} x_{14} -x_{6} x_{9} x_{11} x_{14}^2  -x_{8} x_{11} x_{13} \cr
\qquad   +x_{14}^2  -x_{6} x_{9} x_{11} x_{13} -x_{11}^2 x_{13} x_{14} -x_{5} x_{6} x_{9} 
-x_{5} x_{11} x_{14} -x_{13} x_{14}) \cr
x_{10} - (x_{2} x_{3} -x_{4} x_{5} +x_{11}) 
\end{array}
$$
Thus we have found a tuple of indeterminates~$Z$ and a tuple of polynomials 
which is coherently $Z$-separated. Later on, this will allow us to re-embed
$P/I_F$ and show that this ring is isomorphic to a polynomial ring 
in 10 indeterminates (see Example~\ref{ex:sixteencontinued}).

Now let us try a different choice of~$Z$.
If we choose $x_{10}$ instead of $x_{16}$ in the linear part of~$f_3$, 
then $x_{5}$ instead of $x_{15}$ in the linear part of~$f_4$, 
and $x_{11}$ instead of $x_{10}$ in the  linear part of~$f_6$,
we get $\widetilde{Z} = (x_{12}, x_7, x_{10}, x_5, x_1, x_{11})$.
Applying the algorithm described in Corollary~\ref{cor:checkZ-sep}
in this case allows us to show that~$F$ is not coherently $\widetilde{Z}$-separated.

However, the partial tuple $F' = (f_1, f_2, f_3, f_4, f_5)$ is coherently separated
with respect to $Z'=(x_{12}, x_7, x_{10}, x_5, x_1)$.
The corresponding polynomials in the reduced Gr\"obner basis of~$I_{F'}$ are
\begin{align*}
x_{12} &- (x_{6} x_{9} +x_{11} x_{14})\\
x_{7} \ &- (x_{6} x_{9} +x_{8} x_{13})\\
x_{10} &-  (-x_{6} x_{9} x_{13} -x_{8} x_{13}^2  -x_{6} x_{9} x_{14} -x_{11} x_{14}^2  +x_{16})\\
x_{5} \ &-( -x_{6} x_{9} x_{14} -x_{8} x_{13} x_{14} -x_{13} x_{16} +x_{15})\\
x_{1} \ &- (-x_{6} x_{9} x_{15} -x_{11} x_{14} x_{15} -x_{13} x_{14})
\end{align*}
Later we will see that these polynomials result in a less useful re-embedding
of $P/I_F$ (see Example~\ref{ex:sixteencontinued}).
\end{example}

\bigskip\bigbreak
%%%%%%%%%%%%%%%%%%%%%%%%%%%%%%%%%%%%%%%%%%%%%%%%%%%%%%%%
%
% Section 4: Z-Separating Re-Embeddings
%
%%%%%%%%%%%%%%%%%%%%%%%%%%%%%%%%%%%%%%%%%%%%%%%%%%%%%%%%%

\section{$Z$-Separating Re-Embeddings}
\label{sec:Z-Separating Re-Embeddings}

In this section we apply the material developed in the previous sections 
in order to find good re-embeddings of affine schemes, extending and building on
the investigation of this topic in~\cite{KLR2}. 

Let $P=K[x_1,\dots,x_n]$ be a polynomial ring
over a field~$K$, and let~$I$ be an ideal in~$P$.
Our goal is to find a tuple~$Z$ of indeterminates in $\{x_1,\dots,x_n\}$
which can be eliminated easily, without resorting to potentially very heavy
Gr\"obner basis computations. In other words, we look for an isomorphism
of $K$-algebras $P/I \cong \widehat{P} / (I \cap \widehat{P})$,
where $\widehat{P}=K[Y]$ is the polynomial ring in the remaining indeterminates.

As in the previous sections, we let $\M = \langle x_1,\dots,x_n\rangle$,
we denote the tuple $(x_1,\dots,x_n)$ by~$X$, we let $Z=(z_1,\dots,z_s)$ 
be a tuple of distinct indeterminates in~$P$, and we assume that~$I$ is 
an ideal in~$P$ which is contained in~$\M$.
Recall that, given $f_1,\dots,f_s\in \M\setminus \{0\}$, 
a term ordering~$\sigma$ on~$\mathbb{T}^n$ is called a $Z$-separating term ordering
for the ideal $I_F=\langle f_1,\dots,f_s\rangle$ if we have
$\LT_\sigma(I_F) = \langle z_1,\dots,z_s\rangle$.
Now we generalize this definition to arbitrary ideals~$I$ contained in~$\M$
as follows.

\begin{definition}\label{def:Z-sepTO}
Let $I$ be an ideal in~$P$ which is contained in~$\M$. 
\begin{enumerate}
\item[(a)] A term ordering~$\sigma$
on~$\mathbb{T}^n$ is called a {\bf $Z$-separating term ordering for~$I$}
if there exist polynomials $f_1,\dots,f_s\in I\setminus \{0\}$ such that~$\sigma$
is a $Z$-separating term ordering for~$I_F = \langle f_1,\dots,f_s\rangle$.

\item[(b)] The ideal~$I$ is called {\bf $Z$-separating} if there exists a $Z$-separating
term ordering~$\sigma$ for~$I$.

\end{enumerate}
\end{definition}

Let us characterize this notion in several ways and show that it is
equivalent to the one defined in~\cite{KLR2}, Definition~2.9.a. 

\begin{proposition}{\bf (Characterization of $Z$-Separating Term 
Orderings)}\label{prop:charZ-sepTO}\\
As above, let $I\subseteq \M$ be an ideal in~$P$, and let $\sigma$ be a term
ordering on~$\mathbb{T}^n$. Then the following conditions are equivalent.
\begin{enumerate}
\item[(a)] The term ordering~$\sigma$ is a $Z$-separating term ordering for~$I$.

\item[(b)] There exist polynomials $f_1,\dots,f_s\in I \setminus \{0\}$ such 
that~$\sigma$ is a $Z$-separating term ordering for $F=(f_1,\dots,f_s)$, i.e.,
such that $\LT_\sigma(f_i)=z_i$ for $i=1,\dots,s$.

\item[(c)] There exist polynomials $f_1,\dots,f_s\in I \setminus \{0\}$ such 
that $\LT_\sigma(f_i)=z_i$ for $i=1,\dots,s$ and $F=(f_1,\dots,f_s)$ is 
coherently $Z$-separating.

\item[(d)] The reduced $\sigma$-Gr\"obner basis of~$I$ has the shape
$\{f_1,\dots,f_s,\, g_1,\dots,g_t\}$ where $F=(f_1,\dots,f_s)$
is a coherently $Z$-separating tuple and $\{ g_1,\dots,g_t\}$
is the reduced $\sigma_Y$-Gr\"obner basis of $I\cap \widehat{P}$.

\end{enumerate}
\end{proposition}

\begin{proof}
To prove that~(a) implies~(c), we let $f_1,\dots,f_s\in I \setminus \{0\}$
such that the ideal $I_F = \langle f_1,\dots,f_s\rangle$ satisfies
$\LT_\sigma(I_F) = \langle Z \rangle$. Now we let $\{\tilde{f}_1,\dots,
\tilde{f}_s\}$ be the reduced $\sigma$-Gr\"obner basis of~$I_F$, and
let $\widetilde{F} = (\tilde{f}_1,\dots,\tilde{f}_s)$. 
By Proposition~\ref{prop:Z-sep}.c, the tuple $\widetilde{F}$ is coherently
$Z$-separating and~$\sigma$ is a $Z$-separating term ordering for~$\widetilde{F}$.

As (c)$\Rightarrow$(b) is obvious, we prove (b)$\Rightarrow$(d) next.
Let $\{\tilde{f}_1,\dots,\tilde{f}_s\}$ be the reduced $\sigma$-Gr\"obner
basis of~$I_F = \langle f_1,\dots,f_s\rangle$. By Proposition~\ref{prop:Z-sep}.d,
the polynomials in~$\widetilde{F} = (\tilde{f}_1,\dots,\tilde{f}_s)$
are obtained by making the polynomials in~$F$ monic and interreducing them.
Letting $\{g_1,\dots,g_t\}$ be the reduced $\sigma_Y$-Gr\"obner basis of
$I\cap K[Y]$, it follows that $\{\tilde{f}_1,\dots, \tilde{f}_s, g_1,\dots,g_t\}$
is a $Z$-separating Gr\"obner basis of~$I$ in the sense of~\cite{KLR2}, 
Definition 2.9.b. Then Proposition 2.12 of~\cite{KLR2} shows that  
the reduced $\sigma$-Gr\"obner basis of~$I$ has the form 
$\{z_1 - h_1, \dots, z_s - h_s,\, g_1,\dots,g_t\}$,
where $h_1, \dots, h_s \in K[Y]$.
Since $(z_1-h_1, \dots, z_s - h_s)$ is coherently $Z$-separating, 
the claim follows.

Finally, we note that (d)$\Rightarrow$(a) is obviously true. 
\end{proof}

The following remark provides us with a plenty of $Z$-separating term orderings
for~$I$.

\begin{remark}
Let $I\subseteq \M$ be an ideal in~$P$, and assume that there
exists a $Z$-separating term ordering~$\sigma$ for~$I$. Then every elimination 
ordering~$\tau$ for~$Z$ is a $Z$-separating term ordering for~$I$, because
the reduced $\sigma$-Gr\"obner basis of~$I$ has the shape given in part~(d)
of the proposition. 
\end{remark}

Using the reduced Gr\"obner basis of~$I$ with respect to a $Z$-separating
term ordering, we can now recall the central definition.

\begin{definition}\label{def:sepembed}
Let $I\subseteq\M$ be an ideal in~$P$, and assume that there exists
a $Z$-separating term ordering~$\sigma$ for~$I$. Let $\{z_1-h_1, \dots,
z_s-h_s,\, g_1,\dots,g_t\}$ be the reduced $\sigma$-Gr\"obner basis of~$I$,
where $h_i,g_j \in \widehat{P} = K[Y]$. Then the $K$-algebra isomorphism
$$
\Phi:\; P/I \;\longrightarrow\; \widehat{P} / (I \cap \widehat{P})
$$
given by $\Phi(\bar{x}_i) = \bar{x}_i$ for $x_i \notin Z$ and 
$\Phi(\bar{x}_i) = \bar{h}_j$ for $x_i=z_j\in Z$
is called the {\bf $Z$-separating re-embedding of~$I$}.
\end{definition}

\begin{remark}\label{rem:indepchoice}
In~\cite{KLR2}, Proposition~2.14, it is shown that the isomorphism~$\Phi$
is, in fact, independent of the choice of a $Z$-separating term ordering for~$I$,
as long as one exists. It is important to note that the tuple $F=(f_1,\dots,f_s)$, 
whose existence is required in the definition of a $Z$-separating term
ordering for~$I$, may change during the construction of the isomorphism~$\Phi$.
\end{remark}

In actual fact, the computation of a reduced Gr\"obner basis will be out of
reach in non-trivial cases, and the true way to construct~$\Phi$ will be:
\begin{enumerate}
\item[(1)] Find a (hopefully large) tuple~$Z$ and a tuple of polynomials
$F=(f_1,\dots,f_s)$ in~$I$ such that $z_i \in \Supp(f_i)$ for $i=1,\dots,s$.

\item[(2)] Using the methods explained in Remark~\ref{usingStronglySep} and 
Section~\ref{sec:Finding Z-Separating Tuples}, show that~$F$ is $Z$-separating
and find a $Z$-separating term ordering~$\sigma$.

\item[(3)] Interreduce the polynomials in~$F$ and make them monic in order to compute
a $Z$-separating system of generators $\{\tilde{f}_1,\dots,\tilde{f}_s\}$
of~$I_F = \langle f_1,\dots,f_s\rangle$ such that $\tilde{f}_i = z_i - h_i$
with $h_i \in K[Y]$ for $i=1,\dots,s$.

\item[(4)] Define $\Phi:\; P/I \cong \widehat{P}/(I\cap\widehat{P})$
using~$\widetilde{F}= (\tilde{f}_1,\dots,\tilde{f}_s)$ and use the
observation that this results in the same map as the usage
of the reduced $\sigma$-Gr\"obner basis of~$I$ by~\cite{KLR2}, 
Theorem 2.13.c.
\end{enumerate}

Let us apply this procedure in some concrete cases.

\begin{example}\label{ex:twoZ}
Let $P = \mathbb{Q}[x, y, z, w]$, and let $I = \langle f_1, f_2, f_3, f_4\rangle $,
where $f_1=x^2 -y -z$, $f_2=x^3 -y -w$, $f_3=x^4 -z -w$, and $f_4= x^2-w$.
\begin{enumerate}
\item[(a)] For $Z = (z, w)$, the tuple $F = (f_1, f_2)$ is clearly 
coherently $Z$-separating. By applying Corollary~\ref{cor:checkZ-sep} and an LP solver 
we find that $u=(1,2,3,4)$ satisfies $u\cdot v > 0$ for all~$v$ in 
$$
\DLog_Z(F) \;=\; \{ (-2,0,1,0),\, (0,-1,1,0),\, (-3,0,0,1),\,(0,-1,0,1) \}.
$$
Next we use~$u$  as the first row of a matrix of size $4\times 4$
defining a term ordering~$\sigma$. Then~$\sigma$ is a $Z$-separating term ordering
for~$I$.

The reduced $\sigma$-Gr\"obner basis of~$I$ turns out to be 
$$
\qquad G = \{  z -x^2 +y,\; w -x^2,\; x^3 -x^2 -y,\;  xy -x^2 +2y,\;  y^2 -5x^2 +16y\}
$$
Hence we get the isomorphism $\Phi: P/I \To K[Y]/(I\cap K[Y])$ where $K[Y] = \QQ[x,y]$ 
and $I\cap K[Y] = \langle x^3 -x^2 -y,\; xy -x^2 +2y,\; y^2 -5x^2 +16y\rangle $.

\item[(b)] Now we choose $F' = (f_1, f_3)$ and note that the tuple $F'$ is not coherently 
$Z$-separating. To check whether~$F'$ is $Z$-separating, we apply
Corollary~\ref{cor:checkZ-sep} again. A call to an LP solver tells
us that $u' = (1,1,3,5)$ satisfies $u' \cdot v >0$ for all~$v$ in
$$ 
\quad\DLog_Z(F') \;=\; \{ (-2,0,1,0),\, (0,-1,1,0),\, (-4,0,0,1),\, (0,0,-1,-1) \}.
$$
Using~$u'$ as the first row of a suitable $4\times 4$ matrix, we define a term
ordering~$\tau$ on~$\mathbb{T}^4$.
Here the reduced $\tau$-Gr\"obner basis of~$I$ turns out to be
\begin{align*}
G' \;=\;   \{ &z -\tfrac{1}{5}y^2 -\tfrac{11}{5}y,\; w -\tfrac{1}{5}y^2 -\tfrac{16}{5}y,\;
x^2 -\tfrac{1}{5}y^2 -\tfrac{16}{5}y,\\  
& xy -\tfrac{1}{5}y^2 -\tfrac{6}{5}y,\;  y^3 +17y^2 -9y  \}
\end{align*}
Altogether, we obtain another isomorphism $\Phi': P/I \To K[Y]/(I\cap K[Y])$.
However, we observe that the reduced $\sigma_Y$-Gr\"obner basis~$G$
and the reduced $\tau_Y$-Gr\"obner basis~$G'$ of $I\cap K[Y]$ are very different.

In particular, in~$G$ we have the polynomial $g = x^3 -x^2 -y$. This allows us 
to increase~$Z$ to $Z' = (z, w, y)$, as we have the $Z'$-separating tuple $(f_1, f_2, g)$.
An even better re-embedding of~$P/I$ results. With $G'$ nothing like this happens.
\end{enumerate}
\end{example}

\begin{example}\label{ex:sixteencontinued}
Let us go back to Example~\ref{ex:sixteen}. In $P=\mathbb{Q}[x_1,\dots,x_{16}]$,
we studied a tuple of polynomials $F=(f_1,\dots,f_6)$ and the ideal
$I=\langle f_1,\dots,f_6\rangle$ they generate.
\begin{enumerate}
\item[(a)] Using the tuple $Z = (x_{12}, x_7, x_{16}, x_{15}, x_1, x_{10})$,
we found a vector $u = (11, 1,1,1,1,1, 3,1,1, 3,1,3,1,1, 7,5)$ 
such that every term ordering~$\sigma$ represented by a matrix whose first row is~$u$
has the reduced $\sigma$-Gr\"obner basis described in Example~\ref{ex:sixteen}. 
Using this Gr\"obner basis, the corresponding $Z$-separating re-embedding of~$I$ is
$\Phi_Z:\; P/I \longrightarrow \widehat{P}$, where we have
$\widehat{P} =\QQ[x_2, x_3, x_4, x_5, x_6, x_8, x_9, z_{11}, x_{13},x_{14}]$.

\item[(b)] Instead, the choice $Z' = (x_{12}, x_7, x_{10}, x_5, x_1)$ leads
to the re-embedding $\Phi_{Z'}:\; P/I \longrightarrow \widehat{P}'/J$ with
$\widehat{P}' = \QQ[x_2, x_3, x_4, x_6, x_8, x_9, x_{11}, x_{13}, x_{14}, x_{15}, x_{16}]$
and the ideal $J = I\cap\widehat{P}' = \langle f\rangle$ generated by the polynomial
\begin{align*}
f \;=\; &x_{4} x_{6} x_{9} x_{14} +x_{4} x_{8} x_{13} x_{14} +x_{6} x_{9} x_{13} 
+x_{8} x_{13}^2  +x_{6} x_{9} x_{14}\\ 
&+x_{11} x_{14}^2  +x_{4} x_{13} x_{16} +x_{2} x_{3} -x_{4} x_{15} +x_{11} -x_{16}
\end{align*}
Since~$f$ is neither separating with
respect to~$x_{11}$ nor with respect to~$x_{16}$, it is not easy to see that
$\widehat{P}'/J$ is in fact isomorphic to a polynomial ring. In this sense, the
re-embedding via~$Z$ is better than the re-embedding via~$Z'$.
\end{enumerate}
\end{example}

In view of these examples, the choice of a set~$Z$ of maximal cardinality such that~$I$
contains a $Z$-separating tuple becomes a central question. 
In Section~\ref{sec:Computing} we will come back to this point.

\bigskip\bigbreak
%%%%%%%%%%%%%%%%%%%%%%%%%%%%%%%%%%%%%%%%%%%%%%%%%%%%%%%%
%
% Section 5: Z-Restricted Groebner Fans
%
%%%%%%%%%%%%%%%%%%%%%%%%%%%%%%%%%%%%%%%%%%%%%%%%%%%%%%%%%

\section{$Z$-Restricted Gr\"obner Fans}
\label{sec:Z-Restricted Groebner Fans}

In the preceding section we saw that in order to find good re-embeddings 
of a polynomial ideal, we may try to find a tuple of indeterminates~$Z$
and a $Z$-separating term ordering for~$I$. The possible choices of term
orderings leading to distinct reduced Gr\"obner bases of~$I$ are
classified by the Gr\"obner fan of~$I$, introduced and studied first in~\cite{MR}.
Here we want to classify the possible choices of $Z$-separating term orderings
leading to distinct reduced Gr\"obner bases of~$I$.

First things first, let us recall the setting and the general definitions.
Let $P=K[x_1,\dots,x_n]$ be a polynomial ring over a field~$K$, let 
$\M = \langle x_1,\dots,x_n\rangle$, let $I\subseteq \M$ be a non-zero ideal
of~$P$, and let $Z=(z_1,\dots,z_s)$ be a tuple of distinct indeterminates
$z_i \in \{x_1,\dots,x_n\}$ for $i=1,\dots,s$.

Moreover, recall that, given a term ordering~$\sigma$ and the reduced 
$\sigma$-Gr\"obner basis $G = \{ g_1,\dots,g_r\}$ of~$I$, the set of pairs
$$
\overline{G} \;=\; \{\,  (\LT_\sigma(g_1),\, g_1), \dots, (\LT_\sigma(g_r),\, g_r) \,\} 
$$
is called the {\bf marked reduced $\sigma$-Gr\"obner basis} of~$I$. The set
of all distinct marked reduced $\sigma$-Gr\"obner bases of~$I$ is called
the {\bf Gr\"obner fan} of~$I$ and is denoted by~$\GFan(I)$.
Next we restrict our attention to those marked reduced $\sigma$-Gr\"obner bases
of~$I$ which are $Z$-separating.

\begin{definition}\label{def:Z-GFan}
Let $Z=(z_1,\dots,z_s)$ be a tuple of distinct indeterminates such that
$z_i \in \{x_1,\dots,x_n\}$ for $i=1,\dots,s$. The set of all marked
reduced $\sigma$-Gr\"obner bases 
$\overline{G} = \{ (\LT_\sigma(g_1),g_1),\dots, (\LT_\sigma(g_r), g_r)\}$
of~$I$ such that for $j=1,\dots,s$ there are indices $i_j \in \{1,\dots,r\}$ with
$z_j = \LT_\sigma(g_{i_j})$ is called the {\bf $Z$-restricted
Gr\"obner fan} of~$I$ and is denoted by $\GFan_Z(I)$.
\end{definition}

Notice that, by Proposition~\ref{prop:Z-sep}.d, every choice of a 
$Z$-separating term ordering~$\sigma$ for~$I$
corresponds to a reduced $\sigma$-Gr\"obner basis of~$I$
of the type described in the definition,
but in general not in a unique way. It is also possible to
require that some the leading terms of a reduced Gr\"obner basis
are predescribed terms given by a tuple $T=(t_1,\dots,t_s)$ with $t_i\in \mathbb{T}^n$
rather than predescribed indeterminates given by a tuple $Z=(z_1,\dots,z_s)$
with $z_i \in \{x_1,\dots,x_n\}$. This more general type of restriction
leads to the $T$-restricted Gr\"obner fan introduced in Section~\ref{sec:T-RestrGFan}.

For now, let us have a look at a concrete example of a $Z$-restricted 
Gr\"obner fan.

\begin{example}
Extending Example~\ref{ex:SumAndDiff}, we let $P=\mathbb{Q}[x,y,z]$, let $Z=(x,y)$, and
let $I=\langle f_1,f_2,f_3\rangle$, where $f_1 = x+y - z^2$ and $f_2 = x-y + z^2$
and $f_3= z - y^2$.
If we let $g_1 = \frac{1}{2}\, (f_1+f_2) = x$ and $g_2 = \frac{1}{2}\, (f_1 -f_2)
= y - z^2$ then the tuple $(g_1,g_2)$ is $Z$-separated and the reduced
$\sigma$-Gr\"obner basis of~$I$ for any term ordering~$\sigma$
satisfying $y >_\sigma z^2$ is $G= \{g_1,\, g_2,\, z^4-z\}$.
Here $\GFan_Z(I)$ consists of a single marked reduced Gr\"obner basis, namely
$\overline{G} \;=\; \{\, (x, x),\, (y, y-z^2),\, (z^4, z^4-z) \,\}$.
\end{example}

To clarify the structure of $\GFan_Z(I)$ in general, we introduce
the following map.

\begin{definition}
Let $Z=(z_1,\dots,z_s)$ be a tuple of distinct indeterminates
in $\{x_1,\dots,x_n\}$, and let $Y=(y_1,\dots,y_{n-s})$, where
$\{y_1,\dots,y_{n-s}\} = \{x_1,\dots,x_n\} \setminus \{z_1,\dots,z_s\}$.
Let $I\subseteq\M$ be an ideal in~$P$ such that $\GFan_Z(I)$
is not empty. Consider the map
$$
\Gamma_Z:\;  \GFan_Z(I) \;\longrightarrow \GFan(I\cap K[Y])
$$
given as follows: for a $Z$-separating reduced $\sigma$-Gr\"obner
basis $G=\{g_1,\dots,g_r\}$ of~$I$, we have $r\ge s$ and we may assume 
$\LT_\sigma(g_i)=z_i$ for $i=1,\dots,s$. Next, we let
$$
\Gamma_Z(\overline{G}) \;=\; \{\, (\LT_\sigma(g_{s+1}),\, g_{s+1}),\; \dots,\;
(\LT_\sigma(g_r),\, g_r) \,\}
$$
Then the map $\Gamma_Z$ is called the {\bf $Z$-restricted GFan projection}.
\end{definition}

Notice that, in the setting of this definition, the set $\{g_{s+1},
\dots,g_r\}$ is the reduced $\sigma$-Gr\"obner basis of $I\cap K[Y]$
by Proposition~\ref{prop:charZ-sepTO}.d. Thus the map~$\Gamma_Z$ is well-defined.
Our main result in this section is that the map~$\Gamma_Z$ is bijective.
For its proof we need the following result which is part of the 
theory of Gr\"obner fans (see~\cite{MR}). For the convenience of the readers,
we include it here together with a proof.

\begin{lemma}\label{lem:sameGB}
Let $I$ be an ideal in $P=K[x_1,\dots,x_n]$, let~$\sigma$ and~$\tau$ be two term 
orderings on~$\mathbb{T}^n$, let $G=\{g_1,\dots, g_r\}$ be the reduced 
$\sigma$-Gr\"obner basis of~$I$, and let $G' = \{g'_1, \dots, g'_{r'}\}$ 
be the reduced $\tau$-Gr\"obner bases of~$I$.
If we have $r=r'$ and $\LT_\sigma(g_i) = \LT_\tau(g'_i)$ for $i=1,\dots,r$, then $G = G'$.
\end{lemma}

\begin{proof}
For a contradiction, we may assume that there exists an index $i\in \{1,\dots,r\}$ 
such that $g_i\ne g'_i$. Then, if we let $f = g_i - g'_i$, we have $f\in I \setminus \{0\}$
and no term of~$\Supp(f)$ is divisible by any leading term $\LT_\sigma(g_k)=\LT_\sigma(g'_k)$
with $k\in \{1,\dots,r\}$, because $G$ and~$G'$ are fully interreduced.
Therefore $\LT_\sigma(f) \notin \{ \LT_\sigma(g_1),\dots, \LT_\sigma(g_r)\}$
contradicts the hypothesis that~$G$ is a Gr\"obner basis of~$I$.
\end{proof}

Now we are ready to formulate and prove our main result.

\begin{theorem}{\bf (Structure of Restricted Gr\"obner Fans)}\label{thm:restrGFan}\\
Let $Z=(z_1,\dots,z_s)$ be a tuple of distinct indeterminates
in $\{x_1,\dots,x_n\}$, and let $Y=(y_1,\dots,y_{n-s})$ be such that
$\{y_1,\dots,y_{n-s}\} = \{x_1,\dots,x_n\} \setminus \{z_1,\dots,z_s\}$.
Let $I \subseteq \M$ be an ideal of~$P$ such that the set $\GFan_Z(I)$ is not empty.
Then the map $\Gamma_Z:\; \GFan_Z(I)\longrightarrow \GFan(I\cap K[Y])$ is bijective.
\end{theorem}

\begin{proof} First we prove injectivity. Let~$\sigma$ and~$\tau$ be
$Z$-separating term orderings for~$I$, write $G=\{g_1,\dots,g_r\}$
for the reduced $\sigma$-Gr\"obner basis of~$I$, and write
$H=\{h_1,\dots,h_{r'}\}$ for the reduced $\tau$-Gr\"obner basis of~$I$.
Here we may assume that $r\ge s$ and $r'\ge s$ and
$\LT_\sigma(g_i) = \LT_\tau(h_i) = z_i$ for $i=1,\dots,s$.
Suppose that the map $\Gamma_Z$ satisfies
\begin{align*}
\Gamma_Z(\overline{G}) &\;=\; \{\, (\LT_\sigma(g_{s+1}), g_{s+1}),\; \dots,\; 
(\LT_\sigma(g_r),g_r) \,\}\\ 
&\;=\; \{\, \LT_\tau(h_{s+1}), h_{s+1}),\; \dots,\;
(\LT_\tau(h_{r'}),h_{r'}) \,\} \;=\; \Gamma_Z(\overline{H})
\end{align*}
Then we see that $r=r'$ and $\LT_\sigma(g_i) = \LT_\tau(h_i)$
for $i=s+1,\dots,r$, after possibly renumbering $h_{s+1},\dots,h_r$
suitably. By the lemma, we get $G=H$, and hence $\overline{G}=\overline{H}$.

Next we prove surjectivity. Let~$\tau$ be a term ordering on $K[Y]$, 
and let~$H = \{h_1,\dots,h_t\}$ be the reduced $\tau$-Gr\"obner basis of $I\cap K[Y]$.
Using the theory of Gr\"obner fans (cf.~\cite{MR}), we may assume that
$\tau = {\rm ord}(W)$ with a matrix $W\in \Mat_{n-s}(\mathbb{Z})$.

By assumption, the set $\GFan_Z(I)$ is not empty. 
So, let $F=(f_1,\dots,f_s)$ be a coherently $Z$-separating tuple of polynomials in~$I$.
Now we construct a matrix $V\in \Mat_n(\mathbb{Z})$ which defines a
term ordering $\sigma = {\rm ord}(V)$ in~$\mathbb{T}^n$ as follows.
First we put the matrix of size $s\times s$ defining the ordering 
${\tt DegRevLex}$ on~$K[Z]$ into the columns of~$V$ corresponding to~$Z$ 
(see~\cite{KR1}, Definition 1.4.7 and Proposition 1.4.12)
and zeroes into the other columns. This fills the first~$s$ rows of~$V$.
In the lower $n-s$ rows, we put~$W$ into the columns of~$V$ corresponding to~$Y$
and zeroes into the columns corresponding to~$Z$.

The resulting term ordering~$\sigma$ satisfies $\LT_\sigma(f_i)= z_i$
for $i=1,\dots,s$ and $\LT_\sigma(h_j) = \LT_\tau(h_j)$ for $j=1,\dots,t$.
Therefore the reduced $\sigma$-Gr\"obner basis of~$I$ is of the form
$G=\{ z_1 - g_1,\dots,z_s - g_s,\, h_1,\dots,h_t\}$ with $g_1,\dots,g_s \in K[Y]$
and we get $\Gamma_Z(\overline{G}) = \overline{H}$, as desired.
\end{proof}

The proof of this theorem yields the following observation.

\begin{corollary}
In the setting of the theorem,
let $\sigma$ be a $Z$-separating term ordering for~$I$,
and let~$G$ be the reduced $\sigma$-Gr\"obner basis of~$I$.
Then there exists a $Z$-elimination term ordering~$\tau$ such 
that~$G$ is also a reduced $\tau$-Gr\"obner basis of~$I$.
\end{corollary}

\begin{proof}
Let $G=\{g_1,\dots,g_r\}$ and $\Gamma_Z(\overline{G}) =
\{ (\LT_\sigma(g_{s+1}), \, g_{s+1}), \dots, (\LT_\sigma(g_r),\, g_r)\}$.
As in the proof of the theorem, we now construct a term ordering~$\tau$
of the form $\tau = {\rm ord}(V)$ where~$V$ is a block matrix
consisting of a block of size $s\times s$ representing {\tt DegRevLex}
and a block of size $(r-s) \times (r-s)$ representing $\sigma_Y$.
Then~$\tau$ is an elimination ordering for~$Z$ and the reduced $\tau$-Gr\"obner
basis~$H$ of~$I$ satisfies $\Gamma_Z(\overline{H}) = \Gamma_Z(\overline{G})$. 
This yields $H=G$ by the theorem.
\end{proof}

The theorem can simplify the search for an optimal re-embedding
considerably. To demonstrate how this works, we consider the setting 
of Example~4.3 in~\cite{KLR2}.

\begin{example}\label{ex:bestemb2}
Let $P = \mathbb{Q}[x, y, z, u, v]$, and let $I = \langle f_1, f_2, f_3, f_4\rangle $,  
where $f_1 = x^2 +x -z +v$, $f_2 = z^2 -u^2 -u$, $f_3 = u^2 -y +u -v$, and 
$f_4 = x^2 +u^2 -u$.
Recall that using \cocoa\ we may calculate the Gr\"obner fan of~$I$. It 
comprises 462 marked reduced Gr\"obner bases. Moreover, if we choose $Z' = (y, u, v)$, 
then an optimal re-embedding of~$I$ is given by the $\QQ$-algebra isomorphism
$\Phi_{Z'}:\; P/I \longrightarrow \QQ[x, z] / \langle x^4 +2x^2z^2 +z^4 +2x^2 -2z^2\rangle $.
The computation of the Gr\"obner fan of~$I$ is quite demanding, 
so let us see if we can get the same result using a better strategy 
which follows from the theorem.

If we let  $Z=(y, v)$ and use the technique explained in 
Section~\ref{sec:Finding Z-Separating Tuples}, we find that $(f_3, f_1)$ is 
$Z$-separating and every term ordering $\sigma$ 
represented by a matrix whose first row is $(1,  4,  1,  1,  3)$ 
yields $\{y -x +z -2u,\; v -u^2 +x -z +u \}$ as the reduced $\sigma$-Gr\"obner 
basis of $\langle f_1, f_3\rangle $. 
Here we have $K[Y] = \QQ[x, z, u]$ and the bijective map 
$\Gamma_Z:\; \GFan_Z(I) \cong \GFan (I\cap K[Y])$.
The computation of $\GFan(I\cap K[Y])$, which takes almost no time, 
returns five elements. They are 
$$
\begin{array}{lll}
\overline{H}_1 &=& \{ (z^2,\ z^2 -u^2 -u),  (x^2,\ x^2 +u^2 -u)\} \cr
\overline{H}_2 &=& \{ (u^2,\  u^2 -z^2 +u),  (x^2, \ x^2 +z^2 -2u) \}  \cr
\overline{H}_3 &=& \{ (u, \ u -\tfrac{1}{2}x^2 -\tfrac{1}{2}z^2),  
                      (x^4,\  x^4 +2x^2z^2 +z^4 +2x^2 -2z^2)\}  \cr
\overline{H}_4 &=& \{ (u^2,\ u^2 +x^2 -u),  (z^2,\  z^2 +u^2 -2u) \} \cr
\overline{H}_5 &=& \{ (u,\  u -\tfrac{1}{2}z^2 -\tfrac{1}{2}x^2),  
                      (z^4,\ z^4 +2x^2z^2 +x^4 -2z^2 +2x^2)\}
\end{array}
$$
Let us consider $H_3 = \{u -\tfrac{1}{2}x^2 -\tfrac{1}{2}z^2,\  
x^4 +2x^2z^2 +z^4 +2x^2 -2z^2\}$ and let us find $\Gamma_Z^{-1}(\overline{H}_3)$. 
Any term ordering $\tau$ on $\QQ[x, z, u]$ represented by a matrix whose first row is $(2,1,5)$
yields $H_3$ as the reduced $\tau$-Gr\"obner basis of $I\cap \QQ[x, z, u]$.
According to the proof of the surjectivity in  Theorem~\ref{thm:restrGFan}, we consider
the term ordering $\sigma$ on $P$ represented by the matrix
$
\left(\begin{array}{llrrr}
0 & 1  &   0 &   0 &   1 \cr
0 & 0  &   0 &   0 &  -1 \cr
2 & 0  &   1 &   5 &   0 \cr
0 & 0  &   0 &  -1 &   0 \cr
0 & 0  &  -1 &   0 &   0
\end{array}
\right)
$
and get the following reduced $\sigma$-Gr\"obner basis of~$I$:
$$
G_3=\{
y -x^2 -x -z^2 +z,\  v +x^2 +x -z,\   u -\tfrac{1}{2}x^2 -\tfrac{1}{2}z^2, \  
 x^4 +2x^2z^2  +2x^2 +z^4 -2z^2\}
$$
Notice that $\Gamma_Z(\overline{G}_3) = \overline{H}_3$. 
The $Z$-separating reduced Gr\"obner basis~$G_3$ yields a
re-embedding $\Phi:\; P/I \cong \QQ[x,z,u]/ 
\langle u -\tfrac{1}{2}x^2 -\tfrac{1}{2}z^2,\; x^4 +2x^2z^2 +z^4 +2x^2 -2z^2 \rangle $ 
given by $\Phi(\bar{x}) = \bar{x}$, $\Phi(\bar{z}) = \bar{z}$, $\Phi(\bar{u}) = \bar{u}$, 
$\Phi(\bar{y}) = \bar{x}^2 +\bar{x} +\bar{z}^2 -\bar{z}$, 
and $\Phi(\bar{v}) = -\bar{x}^2 -\bar{x} +\bar{z}$.

With similar computations we get 
$\Gamma_Z(\overline{G}_1) = \overline{H}_1$, $\Gamma_Z(\overline{G}_2) = \overline{H}_2$,
$\Gamma_Z(\overline{G}_3) = \overline{H}_3$, $\Gamma_Z(\overline{G}_4) = \overline{H}_4$,
and $\Gamma_Z(\overline{G}_5) = \overline{H}_5$
where 
\begin{align*}
G_1 = \{ & v +x -z -u^2 +u,\;  y -x +z -2u,\; z^2 -u^2 -u,  \;  x^2 +u^2 -u  \}\\
G_2 = \{ & v +x -z^2 +2u -z,\;  y -x -2u +z,\; u^2 -z^2 +u, \;  x^2 +z^2 -2u \}\\
G_3 = \{ & y -x^2 -x -z^2 +z,\;  v +x^2 +x -z,\; u -\tfrac{1}{2}x^2 -\tfrac{1}{2}z^2,\\  
         & x^4 +2x^2z^2 +2x^2 +z^4  -2z^2  \}\\
G_4 = \{ &v +x^2 -z +x,\;  y +z -2u -x,\;  u^2 +x^2 -u,\;  z^2 +x^2 -2u \}\\
G_5 = \{ &v +x^2 -z +x,\;  y -z^2 -x^2 +z -x,\\
         &u -\tfrac{1}{2}z^2 -\tfrac{1}{2}x^2,\;   z^4 +2x^2z^2 +x^4 -2z^2 +2x^2 \}
\end{align*}
are the five reduced Gr\"obner bases corresponding to the elements of $\GFan_Z(I)$.

Next we note that among the five marked Gr\"obner bases of $\GFan(I\cap \QQ[x, z, u])$ 
there are two which can be used to enlarge~$Z$. They are~$H_3$ and~$H_5$.
Using~$H_3$, for instance, we get $Z' = (y, u, v)$,  
$F'= (f_3, \ u -\tfrac{1}{2}x^2 -\tfrac{1}{2}z^2,\  f_1)$
and the $\QQ$-algebra isomorphism $\Phi':\; P/I \longrightarrow \QQ[x, z] / 
\langle x^4 +2x^2z^2 +z^4 +2x^2 -2z^2\rangle $.
This is exactly the isomorphism we found using the computation of the huge 
$\GFan(I)$  at the beginning of this example.
\end{example}

\bigskip\bigbreak
%%%%%%%%%%%%%%%%%%%%%%%%%%%%%%%%%%%%%%%%%%%%%%%%%%%%%%%%
%
% Section 6: Computing Some Good Re-Embeddings
%
%%%%%%%%%%%%%%%%%%%%%%%%%%%%%%%%%%%%%%%%%%%%%%%%%%%%%%%%%

\section{Computing Some Good Re-Embeddings}
\label{sec:Computing}

As before, we let $P=K[x_1,\dots,x_n]$ be a polynomial ring over a field~$K$, let
$\M = \langle x_1,\dots,x_n\rangle$, and let $I\subseteq \M$ be an ideal of~$P$.
In view of the results of the preceding sections, the task to find
a good re-embedding of~$I$ can be solved if we find a large tuple~$Z$
of distinct indeterminates such that there exists a $Z$-separating
term ordering for~$I$. Unfortunately, this task does not appear to have
a systematic and uniform solution. In this section we collect a number
of heuristics and approaches which we found useful for constructing
good re-embeddings. In the final part of the section we describe
a criterium which allows us to detect re-embeddings which show that $P/I$
is isomorphic to a polynomial ring, i.e., which show that $\Spec(P/I)$
is isomorphic to an affine space.

Once a candidate tuple $Z=(z_1,\dots,z_s)$ and a tuple $F=(f_1,\dots,f_s)$
have been found such that $f_i\in I$ is $z_i$-separating for $i=1,\dots,s$,
the question whether~$I$ is $Z$-separating can be decided efficiently
using the methods of Section~\ref{sec:Finding Z-Separating Tuples}.
To find such a candidate tuple, the following strategy is sometimes useful.

\begin{remark}\label{rem:choosZstrategy}
Let $I=\langle g_1,\dots,g_r\rangle \subseteq \M$ be an ideal in~$P$.
\begin{enumerate}
\item[(1)] Recall that the {\bf linear part} of~$I$
can be computed via 
$$
\Lin_\M(I) \;=\; \langle (g_1)_1, \dots, \langle (g_r)_1 \rangle_K
$$
where $(g_i)_1$ is the homogeneous component of degree one of~$g_i$ for $i=1,\dots,r$
(cf.~\cite{KLR2}, Proposition~1.6).

\item[(2)] Bring the coefficient matrix of $(g_1)_1,\dots,(g_r)_1$ into
reduced row echelon form. Let $g_1',\dots,g_r'$ be the polynomials
obtained by applying the same reduction steps to $g_1,\dots,g_r$.

\item[(3)] For each $i\in \{1,\dots,r\}$ such that $(g'_i)_1 \ne 0$, choose one
indeterminate in the support of this linear form. Let $Z'=(z_1,\dots,z_s)$ be the tuple
of indeterminates gathered in this way, where $z_j \in \Supp((f_j)_1)$ and
$f_j = g'_{i_j}$ for $j=1,\dots,s$ and $i_j \in \{1,\dots,r\}$.

\item[(4)] Use the subtuple~$Z$ of~$Z'$ consisting of all~$z_j$ such that
$f_j$ is $z_j$-separating and check whether the ideal generated by these
polynomials is $Z$-separating.

\end{enumerate}
\end{remark}

There are several potential drawbacks related to this strategy.
Let us illustrate them with examples.

\begin{example}\label{ex:badchoice}
Let $P= \QQ[x_1, x_2, x_3, x_4, x_5]$, and let $I=\langle g_1,g_2,g_3\rangle$,
where $g_1 = x_1 - x_2^2$, $g_2 = x_1 +x_3^2 +x_4^2 + x_5^2$, and $g_3 = x_2 -x_1^3 + x_3^4$.
Let us apply the above steps.
\begin{enumerate}
\item[(1)] We have $\Lin_\M(I) = \langle x_1,\, x_1,\, x_2 \rangle_K$.

\item[(2)] Here we get $g_1,\, g'_2, g_3$, where $g'_2 = g_2 - g_1$.

\item[(3)] This yields $Z'=(x_1,x_2)$ and $f_1=g_1$ as well as $f_3=g_3$.

\item[(4)] We have $Z=Z'$, but unfortunately the ideal $\langle f_1, f_3\rangle =
\langle x_1-x_2^2,\; x_2 - x_1^3 + x_3^4 \rangle$ is not $Z$-separating,
since we would need a term ordering~$\sigma$ such that
$x_1 >_\sigma x_2^2 >_\sigma (x_1^3)^2 = x_1^6$.
\end{enumerate}

Instead, if we perform the Gau\ss-Jordan reduction in Step~(2) differently,
we may calculate $g'_1,g_2,g_3$ where $g'_1 = g_1 - g_2$.
Then Step~(3) yields $Z'=(x_1,x_2)$ and $(f_1,f_2) = (g_2,g_3)$.
Thus we have $Z=Z'$ in Step~(4), and $\langle f_1,f_2\rangle$
turns out to be $Z$-separating.
Altogether, we see that the success of the above strategy may depend
in a subtle way on how the reduction steps are performed in the Gau\ss -Jordan
algorithm.
\end{example}

In the following example we encounter a different problem.

\begin{example}
Let $P = \mathbb{Q}[x,y,z]$ and $I = \langle g_1,\, g_2 \rangle $, where
$g_1 = y + z - x^2$ and $g_2 = x - y^2$.
\begin{enumerate}
\item[(a)] Following the above strategy, we have $g'_1 = g_1$ and $g'_2 = g_2$.
If we now choose $Z = Z' = (y,x)$ then the ideal $\langle g'_1,\, g'_2\rangle =I$
is not $Z$-separating, since we would need $y >_\sigma x^2$ and $x>_\sigma y^2$.

\item[(b)] However, if we choose $Z = Z' = (z,x)$ then~$I$ is $Z$-separating
and $(f_1,\, f_2) = (z-(y^4-y),\,  x - y^2)$ is a coherently $Z$-separated
system of generators of~$I$.
\end{enumerate}
By choosing $Z=(z,x)$, we get the good re-embedding $P/I \cong \mathbb{Q}[y]$.
How do we arrive at this choice? One way would be to try each indeterminate
in $\Supp(g'_i)_1$ in Step~(3) above. Unfortunately, this may result
in too many cases for non-trivial examples. 
\end{example}

To overcome the difficulty presented by the preceding example,
one method is to {\it shuffle} the indeterminates which show up in the supports 
of the linear forms. We found that, after a few shuffles, the correct choice was detected
quite frequently. The next example is a case in point.

\begin{example}\label{ex:drawbacks}
Let $P = \QQ[x,y,z,w]$, and let $I = \langle g_1,\, g_2,\, g_3\rangle $, where 
$g_1 = x+z+w -y^2$, $g_2 = y +w -x^2$, and $g_3 = y -xw$. 
\begin{enumerate}
\item[(a)] Following Remark~\ref{rem:choosZstrategy}, we find
that $g'_1 = x +z +x^2 -y^2 -xw$, $g_2' = w -x^2 +xw$, and $g_3' = y -xw$
have linear form which correspond to a reduced row echelon form of their coefficient
matrix. This suggests to use $Z=(x,w,y)$ or $Z=(z,w,y)$, but none of these
choices work.

\item[(b)] After shuffling the indeterminates in the supports of $(g'_1)_1$, 
$(g'_2)_1$, and $(g'_3)_1$, we find that $Z=(z,y)$ leads to the $Z$-separating ideal 
$\langle g'_1,\, g'_3\rangle$ and the coherently $Z$-separating tuple 
$(z - (x^2w^2 -x^2 +xw -x),\, y - xw)$. This allows us to construct the re-embedding 
$P/I \cong \mathbb{Q}[x,w] / \langle x^2 -xw -w \rangle$.

\end{enumerate}
\end{example}

In the following example we combine several techniques to find an optimal
re-embedding of an ideal in~$P$. Recall that a re-embedding 
$\Phi_Z:\; P/I \cong \widehat{P} / (I\cap \widehat{P})$ was called
{\bf optimal} in~\cite{KLR2}, Definition~3.3, if any other
re-embedding $P/I \cong P'/I'$ satisfies $\dim(P')\ge \dim(\widehat{P})$.

\begin{example}\label{ex:singular}
Let $P = \QQ[x,y,z]$, and let $I = \langle f_1, f_2, f_3\rangle $, 
where $f_1=x^2 -y^2 -x +2y -z +1$, $f_2=y^3 -3y^2 +3y -z +1$, and 
$f_3=y^2 +x -2y +z -2$. Our goal is to find an optimal re-embedding
of~$I$. 

Note that the ideal~$I$ is not contained in~$\M$. Therefore, as explained
in the introduction, we have to perform a linear change of coordinates
such that the origin is in $\mathcal{Z}(I)$, and the most useful
way to do this is to move a singularity of $\mathcal{Z}(I)$ to the origin.
In the current example, the ring $P/I$ has a unique singular 
point at $p=(1,1,2)$. Therefore we perform the linear change of coordinates
$x \mapsto x+1$, $y\mapsto y+1$, $z \mapsto z+2$ and get the new
ideal $I' = \langle f'_1, f'_2, f'_3\rangle$, where $f'_1=x^2 -y^2 +x -z$,
$f'_2= y^3 -z$, and $f'_3 = y^2 +x +z$.

The linear part of this ideal is $\Lin_\M(I') = \langle x-z, z, x+z \rangle$.
As discussed above, for choosing good candidates for the tuple~$Z$,
we have to bring these linear parts into reduced row echelon form.
Here it turns out that one of the three linear forms reduces to zero,
and that it is best to reduce the first one to zero via $f'_1 \mapsto 
f'_1 - f'_3 - 2f'_2$.
Then we reduce~$f'_3$ to $f'_3 +f'_2$ and get
$I' = \langle g_1, g_2, g_3 \rangle$ where
$g_1 = f'_2 = y^3 -z$, $g_2 = f'_3 +f'_2 = y^3 + y^2 + x$, and
$g_3 = f'_1 -f'_3 - 2 f'_2 = -2y^3 -2y^2 + x^2$.

Here $(g_1,g_2)$ is coherently $Z$-separating with respect to
$Z=(z,x)$. This leads to the re-embedding $\Phi_Z:\; 
P/I' \cong \QQ[y] / \langle y^6 + 2y^5 + y^4 -2y^3 -2y^2 \rangle$.
Using Corollary 4.2 of~\cite{KLR2}, we deduce that this is an 
optimal re-embedding.

Going back to the original coordinate system, we see that the map given by 
$z \mapsto y^3 -3y^2 +3y +1,\  x \mapsto -y^3 +2y^2 -y +1$ provides a 
$\QQ$-algebra isomorphism  
$P/I\cong K[y] / \langle y^6 -4y^5 +6y^4 -6y^3 +5y^2 -2y\rangle$
which is an optimal re-embedding.
\end{example}

As promised, the final part of the section provides a positive result
related to the search of good re-embeddings.
According to Remark~\ref{rem:indepchoice},
if there exists a tuple of~$s$ distinct indeterminates $Z=(z_1,\dots,z_s)$ 
in $X = (x_1,\dots,x_n)$ such that~$I$ is $Z$-separating, 
then every $Z$-separating term ordering~$\sigma$ for~$I$ produces the same isomorphism 
$\Phi:\; P/I \;\longrightarrow\; \widehat{P} / (I \cap \widehat{P})$, as
described in Definition~\ref{def:sepembed}.
Additionally, recall that Corollary 4.2 of~\cite{KLR2} says that, if $\dim_K(\Lin_\M(I)) =s$,
then the isomorphism~$\Phi$ is an optimal embedding. Unfortunately, 
the following example shows that we cannot reverse this implication.

\begin{example}\label{ex:optbiggerthans}
Let $P =\QQ[x, y]$, and let $I = \langle x - y^2,\, y - x\rangle$. In this case
we have $\M =\langle x, y \rangle$ and $\Lin_\M(I) = \langle x, y \rangle$.
For the tuple $Z=(x)$, we have $s=1$ and the inequality 
$\dim_{\QQ}(\Lin_\M(I)) =2 > s=1$. Nevertheless, the isomorphism
$\Phi: P/I \To \mathbb{Q}[y]/\langle y-y^2 \rangle =: R$ is an optimal re-embedding,
since $P/I$ cannot be isomorphic to~$\QQ$, as~$R$ is not an integral domain.
\end{example}

However, the next result says that we can sometimes find good re-embeddings
which show that our original scheme is isomorphic to an affine space.

\begin{proposition}{\bf (Affine Space Criterion)}\label{prop:AffineSpace}\\
Let $Z = (z_1,\dots,z_s)$ be a tuple of distinct indeterminate in $X=(x_1,\dots,x_n)$,
and let $I \subseteq\M$ be an ideal such that the following conditions are
satisfied.
\begin{enumerate}
\item[(1)] The ideal~$I$ is $Z$-separating.

\item[(2)] The ring $P_{\M} / I_{\M}$ is a regular local ring.

\item[(3)] The $K$-vector space $\Lin_{\M}(I)$ is $s$-dimensional.
\end{enumerate}
Then the following claims hold.
\begin{enumerate}
\item[(a)] We have $s = n - \dim(P_{\M}/I_{\M})$.

\item[(b)]  We have $I \cap \widehat{P} = \{0\}$. 
\end{enumerate}
In particular, the map $\Phi:\; P/I \longrightarrow \widehat{P} = K[X\setminus Z]$ 
is an isomorphism with a polynomial ring and the scheme $\Spec(P/I)$ is isomorphic 
to an $(n-s)$-dimensional affine space over~$K$.
\end{proposition}

\begin{proof}
To prove (a), we let $d = \dim(P_{\M}/I_{\M})$. 
Since $P_{\M}/I_{\M}$ is a regular local ring, the cotangent space at its
maximal ideal has dimension~$d$. Hence~\cite{KLR2}, Proposition~1.5.b 
and hypothesis~(3) imply $s = n - d$.

Now we show~(b). Using~(a), we deduce that 
the cardinality of $X \setminus Z$ is $n-(n-d)=d$.
We know that the map $\Phi$ is an isomorphism between $P/I$ and
$\widehat{P}/(I \cap \widehat{P})$, and that $\widehat{P}$ 
is a polynomial ring having~$d$ indeterminates.
On the other hand, the dimension of the localization $\dim(P_\M / I_\M)$ of~$P/I$
is invariant under this isomorphism, and hence we have 
$\dim(\widehat{P} / (I\cap \widehat{P})) \ge d$. 
It follows that $I\cap \widehat{P}$ has to be the zero ideal.
\end{proof}

\bigskip\bigbreak
%%%%%%%%%%%%%%%%%%%%%%%%%%%%%%%%%%%%%%%%%%%%%%%%%%%%%%%%
%
% Section 7: T-Restricted Groebner Fans
%
%%%%%%%%%%%%%%%%%%%%%%%%%%%%%%%%%%%%%%%%%%%%%%%%%%%%%%%%%

\section{$T$-Restricted Gr\"obner Fans}
\label{sec:T-RestrGFan}

The definition of the $Z$-restricted Gr\"obner fan of an ideal can be generalized. 
To introduce the idea, let us look at a simple example.

\begin{example}\label{ex:TGFan}
Let $P = \mathbb{Q}[x, y, z, w]$, let $F=(f_1, f_2)$, where $f_1= x^2 -yz -xw$ 
and $f_2= y^3 -w^3 -x$, and let $I = \langle f_1,\,f_2\rangle $.
Then the Gr\"obner fan of~$I$ consists of 24 marked reduced Gr\"obner bases.
The only way to choose~$Z$ to get a non-trivial restricted Gr\"obner fan
is $Z=(x)$, and in this case we have a 3-element set $\GFan_Z(I) \cong 
\GFan(I \cap \mathbb{Q}[x,y,w])$. Its three elements are
\begin{align*}
\{(x,\ x -y^3 +w^3), ( y^6, \ y^6  -2y^3w^3  -y^3 w  +w^6  +w^4  -yz) \}\\
\{(x,\ x +w^3 -y^3), (yz,\  yz -w^6 +2y^3w^3 -w^4 -y^6 +y^3w  )\}\\
\{(x,\ x +w^3 -y^3), ( w^6, \ w^6  -2y^3w^3 +w^4 +y^6 -y^3w -yz)\}
\end{align*}

However, if we use $T = (yz, x)$ and require that a marked reduced
$\sigma$-Gr\"obner basis of~$I$ contains polynomials $g_1,g_2$ such that
$\LT_\sigma(g_1)=yz$ and $\LT_\sigma(g_2)=x$, then there exists only
one such Gr\"obner basis, namely
$\{ (yz,\;  yz -w^6 +2y^3w^3 -w^4 -y^6 +y^3w),\, (x,\; x +w^3 -y^3)\}$.
For instance, it can be obtained  using any term ordering represented by
a matrix whose first row is $(4,1,8,1)$.
\end{example}

This example shows that it makes sense to use~$T$ to define a restricted 
Gr\"obner fan of~$I$. The general definition is given as follows.

\begin{definition}\label{def:T-sep}
Let $I\subseteq \M$ be an ideal in~$P$, and let
$T=(t_1,...,t_s)$ be a tuple of distinct terms in~$\mathbb{T}^n\setminus \{1\}$.
\begin{enumerate}
\item[(a)] A tuple of polynomials $F=(f_1,\dots,f_s)$ with 
$f_i \in\M\setminus \{0\}$ for $i=1,\dots,s$
is called {\bf coherently $T$-separating} if the following conditions
are satisfied.
\begin{itemize}
\item[(1)] There exists a term ordering~$\sigma$ such that we have 
$\LT_\sigma(f_i)=t_i$ for $i=1,\dots,s$.

\item[(2)] For $i,j \in \{1,\dots,s\}$ such that $i\ne j$, no term
in the support of~$f_j$ is divisible by~$t_i$.
\end{itemize}
In this case $\sigma$ is called a {\bf $T$-separating term ordering} for~$F$.

\item[(b)] A term ordering~$\sigma$ is called a {\bf $T$-separating
term ordering} for~$I$ if there
exists a coherently $T$-separating tuple $F=(f_1,\dots,f_s)$ 
with $f_1,\dots,f_s\in I\setminus \{0\}$ such that~$\sigma$ is a $T$-separating
term ordering for~$F$.

\item[(c)] The set of all marked reduced Gr\"obner bases of~$I$ of the form
$$
\overline{G} \;=\; \{\, (\LT_\sigma(g_1),\, g_1),\; \dots, \;
(\LT_\sigma(g_r),\, g_r) \,\}
$$ 
such that for $i=1,\dots,s$ there exist $j_i\in \{1,\dots,r\}$
with $\LT_\sigma(g_{j_i})=t_i$ is called the {\bf $T$-restricted Gr\"obner fan} of~$I$ 
and is denoted by $\GFan_T(I)$.

\end{enumerate}
\end{definition}

Notice that, for the case of a tuple of distinct indeterminates~$T$,
this definition agrees with Definitions~\ref{def:Z-sep}.c, \ref{def:Z-sepTO}, 
and~\ref{def:Z-GFan} if we keep Proposition~\ref{prop:charZ-sepTO}
in mind. Moreover, for a marked reduced Gr\"obner basis~$\overline{G}$
in $\GFan_T(I)$ as in part~(c) of the definition, the tuple
$(g_{j_1},\dots,g_{j_s})$ is coherently $T$-separating.

The following example illustrates the $T$-restricted Gr\"obner fan of an ideal.

\begin{example}\label{ex:T-GFan}
Let $P = \mathbb{Q}[x,y,z,w]$, and let $I=\langle f_1,f_2,f_3\rangle$,
where $f_1= x^2 +yz +xw$, $f_2= y^2-w^2-x$, and
$f_3=w^2-y+w$. Then the Gr\"obner fan of~$I$ consists of 162 marked reduced Gr\"obner 
bases. 
\begin{enumerate}
\item[(a)] For $Z=(x,w)$, the $Z$-restricted Gr\"obner fan of~$I$ is empty,
but if we let $T=(x,w^2)$ then $\GFan_T(I)$ consists of 9 marked
reduced Gr\"obner bases.

\item[(b)] Using $T=(x,yz)$, we get that $\GFan_T(I)$ contains only 
one element, namely $\overline{G} = \{(x,\, g_1),  (w^2,\, g_2 ), (yz,\, g_3) \}$,
where $g_1 = x -y^2 -w +y$, $g_2 = w^2 +w -y$, and $g_3 = 
yz +y^4 +3y^2w -2y^3 -3yw +y^2 -2w +2y$.
This Gr\"obner basis~$G$ is also the reduced Gr\"obner basis of~$I$
with respect to the elimination term ordering ${\tt Elim}(x,z)$.
Therefore we get $I\cap \mathbb{Q}[y,w] = \langle w^2 + w-y\rangle$. 
Notice that this does not yield a re-embedding of~$I$, 
since there is no natural $\QQ$-algebra isomorphism
between $P/I$ and $\mathbb{Q}[y,w]/ \langle w^2 + w-y\rangle$, as we do not know
how to map~$\bar{z}$. 
\end{enumerate}
\end{example}

This example shows that the elements of $\GFan_T(I)$ do not yield re-embeddings
of~$I$ in the sense of Section~\ref{sec:Z-Separating Re-Embeddings}. However,
they allow us to perform certain eliminations and correspond to the following 
type of isomorphisms.

\begin{proposition}
Let $I\subseteq \M$ be an ideal in~$P$, and let
$T=(t_1,...,t_s)$ be a tuple of distinct terms in~$\mathbb{T}^n$ such that
$\GFan_T(I) \ne \emptyset$. Let $Z$ be the set of all indeterminates
in $\{x_1,\dots,x_n\}$ dividing one of the terms in~$T$, and let~$Y$ 
be the complement of~$Z$ in $\{x_1,\dots,x_n\}$.
\begin{enumerate}
\item[(a)] The set of all terms $u\in\mathbb{T}(Z)$ such that 
$u \notin \langle t_1,\dots,t_s\rangle$ forms an order ideal
in~$\mathbb{T}(Z)$ which we denote by~$\mathcal{O}_T$.

\item[(b)] Let $\overline{G} = \{ (t_1, t_1 - h_1), \dots, (t_s, t_s - h_s),
(\LT_\sigma(g_{s+1}),\, g_{s+1}), \dots, (\LT_\sigma(g_r),\, g_r)\}$
be an element of $\GFan_T(I)$, where $h_1,\dots,h_s, g_{s+1},\dots,g_r\in P$. 
Consider the free $K[Y]$-module $M= \bigoplus_{t\in \OO_T}\, t\cdot K[Y]$
and form the set 
$$
H \;=\; \{ u\, g_i \mid u\in \mathbb{T}(Z),\; i\in \{s+1,\dots,r\},\;
\Supp(u\, g_i) \subseteq M \}
$$
Then~$H$ is a (possibly infinite) Gr\"obner basis of the $K[Y]$-submodule 
$I \cap M$ of~$M$ with respect to the
restriction $\sigma_M$ of~$\sigma$ to~$M$.

\item[(c)] In the setting of~(b), let $F=(t_1-h_1, \dots, t_s-h_s)$.
Then the normal form map
$\NF_F:\; P \longrightarrow M$ induces an isomorphism of $K[Y]$-modules
$\Phi_T:\; P/I \cong M / (I \cap M)$
which is called the {\bf $T$-separating module re-embedding} of~$I$.
\end{enumerate}
\end{proposition}

\begin{proof}
To prove~(a), it suffices to note that the complement of the
monoideal generates by $\{t_1,\dots,t_s\}$ in~$\mathbb{T}(Z)$
is an order ideal, i.e., for $u\in \OO_T$ and $v\in \mathbb{T}(Z)$
such that $v\mid u$ we have $v\notin \langle t_1,\dots,t_s\rangle$
and hence $v\in \OO_T$.

To show~(b), we first note that~$\sigma_M$ is clearly a $K[Y]$-module
term ordering on~$M$. Since the polynomials $g_{s+1},\dots,g_r$
are fully reduced against the polynomials $t_1 -h_1,\dots,t_s -h_s$, 
no term in their supports is divisible by a term in~$T$. 
Hence they are contained in~$I\cap M$.
It remains to show that their leading terms generate the leading term 
module of~$I \cap M$. For $f\in I\cap M$, we have $\NF_G(f)=0$.
However, since the support of~$f$ is contained in~$M$, only
$g_{s+1},\dots,g_r$ can be involved in the reduction steps $f \TTo{G} 0$.
We may assume that in each reduction step a polynomial of the
form $u g_i\in M$ with $u\in \mathbb{T}^n$ and $i\in \{s+1,\dots,r\}$
is subtracted. We write $u = u'\, u''$ with $u'\in \mathbb{T}(Z)$
and $u''\in\mathbb{T}(Y)$. Then the reduction step subtracts a $K[Y]$-multiple
of $u' g_i \in M$. Hence~$f$ can be reduced to zero in the $K[Y]$-module
$I\cap M$ via the elements of~$H$, and the claim is proved.

Finally, we note that the map $\NF_F$ in part~(c) is well-defined, because
$\OO_T \cdot \mathbb{T}(Y)$ are the terms in the complement of the
monoideal generated by $t_1,\dots,t_s$. The map $\NF_F$ is 
$K[Y]$-linear, as the indeterminates in~$Y$ do not divide any term in~$T$.
Letting $\epsilon:\; M\longrightarrow M/ (I\cap M)$ be
the canonical surjection, it is clear that the composed
map $\epsilon \circ \NF_F :\; P \longrightarrow M 
\longrightarrow M/(I\cap M)$ is surjective. Hence it suffices to show
that the kernel of $\epsilon\circ \NF_F$ is~$I$. This follows from 
the fact that a polynomial $f\in P$ satisfies $\NF_F(f)\in I$ if and
only if~$f$ reduces to zero via $g_{s+1},\dots,g_r$, and this is equivalent
to $f\in I\cap M$.
\end{proof}

Notice that the $K[Y]$-module~$M$ is not necessarily finitely generated.
To get a full analogue to Theorem~\ref{thm:restrGFan} in the
$T$-separating case, one would have to develop a theory of
Gr\"obner fans for modules and then examine which Gr\"obner bases
of~$I\cap M$ result from the restriction given in part~(b) of the proposition.
We leave this task to the interested readers and end this paper
by applying the proposition to an easy example and to the 
setting of Example~\ref{ex:T-GFan}.b.

\begin{example}\label{ex:easyTGFan}
Let $P = \QQ[x,y]$, let $T =(x^3)$, and let $G =\{ g_1, g_2\}$,
where $g_1 = x^3 -x$ and $g_2 = xy$.
Let us look at the various statements of the proposition.
\begin{enumerate}
\item[(a)] We have $X=(x,y)$ and $Z=(x)$ and $Y=(y)$.
The order ideal $\OO_T \subseteq \mathbb{T}(Z)$ is given by 
$\OO_T = \{1,\, x,\, x^2\}$.

\item[(b)] The module~$M$ is the free $\QQ[y]$-module
$M \;=\; \QQ[y] \oplus x\, \QQ[y] \oplus x^2\, \QQ[y]$.\\
The set $H = \{ g_2,\, xg_2 \} =\{xy, \, x^2y\}$ is a $\sigma_M$-Gr\"obner 
basis of the $\QQ[y]$-module $I \cap M$, and 
hence $I\cap M = \QQ[y] \cdot xy \oplus \QQ[y] \cdot x^2y $.

\item[(c)] The $T$-separating module re-embedding of~$I$ is given
by the $\QQ[y]$-iso\-mor\-phism 
$P/I \cong  (\QQ[y] \oplus x\, \QQ[y] \oplus x^2\, \QQ[y])/
( \QQ[y] \cdot xy \oplus \QQ[y] \cdot x^2y )$.
\end{enumerate}
\end{example}

\begin{example}\label{ex:T-GFan-continued}
Let $P = \QQ[x,y,z,w]$, let $T=(x,yz)$, and let $G= \{ g_1,g_2,g_3 \}$,
where $g_1 = x -y^2 -w +y$, $g_2 = w^2 +w -y$, and $g_3 = 
yz +y^4 +3y^2w -2y^3 -3yw +y^2 -2w +2y$. Here the term ordering~$\sigma$
satisfies $\LT_\sigma(g_1)=x$, $\LT_\sigma(g_2)=w^2$, and $\LT_\sigma(g_3)=yz$.
Thus we have $Z=(x,y,z)$ and $Y=(w)$.
Let us look at the various statements of the proposition.
\begin{enumerate}
\item[(a)] The order ideal $\OO_T$ is given by $\OO_T = \{1,\, y,\,
z,\, y^2,\, z^2,\, \dots \}$.

\item[(b)] The module~$M$ is the free $\QQ[w]$-module
$$
M \;=\;  \QQ[w] \oplus y\, \QQ[w] \oplus z\, \QQ[w] \oplus y^2\, \QQ[w]
\oplus z^2\, \QQ[w] \oplus \cdots
$$
The set $H = \{ g_2,\, yg_2,\, y^2g_2, \dots\}$ is a $\sigma_M$-Gr\"obner 
basis of the $\QQ[w]$-module $I \cap M$, and hence $I\cap M = \QQ[y,w] \cdot g_2$.

\item[(c)] The $T$-separating module re-embedding of~$I$ is given
by the $\QQ[w]$-iso\-mor\-phism $P/I \cong M / \QQ[y,w]\cdot g_2 \cong
\QQ[y,w]/ \langle g_2\rangle \oplus \bigoplus_{i\ge 1}\QQ[w]\, z^i$.

\end{enumerate}
\end{example}

\bigskip\bigbreak
%%%%%%%%%%%%%%%%%%%%%%%%%%%%%%%%%%%%%%%%%%%%%%%%%%%%%%%
%
%   Bibliography
%
%%%%%%%%%%%%%%%%%%%%%%%%%%%%%%%%%%%%%%%%%%%%%%%%%%%%%%%

\end{document}